\documentclass{amsart}

\usepackage[T1]{fontenc}
\usepackage{enumerate, amsmath, amsfonts, amssymb, amsthm, mathrsfs, wasysym, graphics, graphicx, xcolor, url, hyperref, hypcap, xargs, multicol, pdflscape, multirow, hvfloat, array, ae, aecompl, pifont, mathtools, a4wide, float, blkarray}
\usepackage{marginnote}
\hypersetup{colorlinks=true, citecolor=darkblue, linkcolor=darkblue}
\usepackage[all]{xy}
\usepackage{tikz}
\usepackage{tkz-graph}
\usetikzlibrary{trees, decorations, decorations.markings, shapes, arrows, matrix, calc, fit, intersections, patterns, angles}
\graphicspath{{figures/}}
\makeatletter\def\input@path{{figures/}}\makeatother
\usepackage{caption}
\newtheorem{theorem}{Theorem}[section]
\newtheorem{corollary}[theorem]{Corollary}
\newtheorem{proposition}[theorem]{Proposition}
\newtheorem{lemma}[theorem]{Lemma}

\newtheorem*{theorem*}{Theorem}

\theoremstyle{definition}
\newtheorem{definition}[theorem]{Definition}
\newtheorem{example}[theorem]{Example}
\newtheorem{remark}[theorem]{Remark}

\newtheorem{notation}[theorem]{Notation}

\newtheorem{convention}[theorem]{Convention}

\newcommand{\R}{\mathbb{R}} 
\newcommand{\Z}{\mathbb{Z}} 
\renewcommand{\b}[1]{\mathbf{#1}} 

\newcommand{\set}[2]{\left\{ #1 \;\middle|\; #2 \right\}} 
\newcommand{\bigset}[2]{\big\{ #1 \;\big|\; #2 \big\}} 
\newcommand{\ssm}{\smallsetminus} 
\newcommand{\dotprod}[2]{\left\langle \, #1 \; \middle| \; #2 \, \right\rangle} 
\newcommand{\eqdef}{\mbox{\,\raisebox{0.2ex}{\scriptsize\ensuremath{\mathrm:}}\ensuremath{=}\,}} 
\newcommand{\transpose}[1]{{#1}^t} 

\newcommand{\fref}[1]{Figure~\ref{#1}} 
\newcommand{\ie}{\textit{i.e.}~} 
\newcommand{\eg}{\textit{e.g.}~} 
\newcommand{\aka}{\textit{a.k.a.}~} 
\definecolor{darkblue}{rgb}{0,0,0.7} 
\definecolor{green}{RGB}{57,181,74} 
\definecolor{violet}{RGB}{147,39,143} 
\newcommand{\red}{\color{red}} 
\newcommand{\green}{\color{green}} 
\newcommand{\darkblue}{\color{darkblue}} 
\newcommand{\defn}[1]{\textsl{\darkblue #1}} 
\newcommand{\para}[1]{\medskip\noindent\textit{#1.}} 
\newcommand*\circled[1]{\tikz[baseline=(char.base)]{\node[shape=circle,draw,inner sep=1pt, scale=.7] (char) {#1};}}

\usepackage{todonotes}

\newcommand{\blossom}{^\text{\ding{96}}} 
\newcommand{\Enrs}[1]{E_{nr}^{s}(#1)}
\newcommand{\Ers}[1]{E_{r}^{s}(#1)}
\newcommand{\Enrt}[1]{E_{nr}^{t}(#1)}
\newcommand{\Ert}[1]{E_{r}^{t}(#1)}

\newcommand{\peaks}[1]{\mathsf{peaks}(#1)} 
\newcommand{\deeps}[1]{\mathsf{deeps}(#1)} 
\newcommand{\distinguishedWalk}[2]{\mathsf{dw}(#1,#2)} 
\newcommand{\distinguishedArrows}[2]{\mathsf{da}(#1,#2)} 
\newcommand{\distinguishedString}[2]{\mathsf{ds}(#1,#2)} 
\newcommand{\distinguishedSign}[2]{\varepsilon(#1,#2)} 

\newcommand{\kn}{\kappa} 
\newcommand{\KN}{\textsc{kn}} 
\newcommandx{\NKC}[1][1=\bar Q]{\mathcal{K}_{\mathrm{nk}}(#1)} 
\newcommandx{\RNKC}[1][1=\bar Q]{\mathcal{C}_{\mathrm{nk}}(#1)} 
\newcommandx{\NKL}[1][1=\bar Q]{\mathcal{L}_{\mathrm{nk}}(#1)} 
\newcommandx{\NKG}[1][1=\bar Q]{\mathcal{G}_{\mathrm{nk}}(#1)} 
\newcommandx{\NFC}[1][1=\bar Q]{\mathcal{C}_{\mathrm{nf}}(#1)} 
\newcommand{\peak}{\mathrm{peak}} 
\newcommand{\deep}{\mathrm{deep}} 
\renewcommand{\top}{\mathrm{top}} 
\newcommand{\bottom}{\mathrm{bot}} 
\newcommand{\walk}{\operatorname{\omega}} 

\newcommand{\surface}{\mathcal{S}} 
\newcommand{\dual}{^*} 
\newcommand{\dissection}{\mathrm{D}} 
\newcommand{\vertices}{\mathcal{V}} 
\newcommand{\edges}{\mathcal{E}} 
\newcommand{\faces}{\mathcal{F}} 
\renewcommandx{\AC}[1][1=\dissection]{\mathcal{K}_{\mathrm{acc}}(#1)} 
\newcommandx{\RAC}[1][1=\dissection]{\mathcal{C}_{\mathrm{acc}}(#1)} 
\newcommandx{\SC}[1][1=\dissection\dual]{\mathcal{K}_{\mathrm{sla}}(#1)} 
\newcommandx{\RSC}[1][1=\dissection\dual]{\mathcal{C}_{\mathrm{sla}}(#1)} 
\newcommandx{\NCC}[1][1={\dissection, \dissection\dual}]{\mathcal{K}_{\mathrm{nc}}(#1)} 
\newcommandx{\RNCC}[1][1={\dissection, \dissection\dual}]{\mathcal{C}_{\mathrm{nc}}(#1)} 
\newcommand{\curveof}{\operatorname{\gamma}} 
\newcommand{\edgeof}{\operatorname{\varepsilon}} 
\newcommand{\dualedgeof}{\operatorname{\varepsilon}\dual} 
\DeclareRobustCommand{\SSS}{\reflectbox{$\mathsf{Z}$}} 
\DeclareRobustCommand{\ZZZ}{\mathsf{Z}} 
\newcommand{\vnext}[1]{#1_{\operatorname{next}}} 
\newcommand{\vprevious}[1]{#1_{\operatorname{prev}}} 

\newcommand{\dvector}[1]{\mathbf{d}(#1)} 
\newcommand{\dvectors}[1]{\mathbf{d}(#1)} 
\newcommandx{\dvectorFan}[1][1=\bar Q]{\mathcal{F}^\mathbf{d}(#1)} 
\newcommand{\gvector}[1]{\mathbf{g}(#1)} 
\newcommand{\gvectors}[1]{\mathbf{g}(#1)} 
\newcommandx{\gvectorFan}[1][1=\bar Q]{\mathcal{F}^\mathbf{g}(#1)} 
\newcommand{\cvector}[2]{\mathbf{c}(#1 \in #2)} 
\newcommand{\cvectors}[1]{\mathbf{c}(#1)} 
\newcommandx{\allcvectors}[1][1=\bar Q]{\mathbf{C}(#1)} 
\newcommandx{\cvectorFan}[1][1=\bar Q]{\mathcal{F}^\mathbf{c}(#1)} 
\newcommand{\point}[1]{\mathbf{p}(#1)} 
\newcommand{\HS}[1]{\mathbf{H}^{\le}(#1)} 
\newcommandx{\Asso}[2][1=\bar Q,2={}]{\mathsf{Asso}^{#2}(#1)} 
\newcommandx{\Zono}[2][1=\bar Q,2={}]{\mathsf{Zono}^{#2}(#1)} 
\newcommand{\multiplicityVector}{\b{m}} 

\newcommandx{\AR}[1][1=\bar Q]{\mathrm{AR}(#1)} 
\newcommandx{\tTC}[1][1=\bar Q]{\mathcal{K}^{\textrm{s$\tau$-tilt}}(#1)} 

\newcommand{\koszul}{^!} 

\makeatletter
\def\l@section{\@tocline{1}{2pt}{0pc}{}{}}
\makeatother
\let\oldtocpart=\tocpart
\renewcommand{\tocpart}[2]{\bf\large\oldtocpart{#1}{#2}}
\let\oldtocsection=\tocsection
\renewcommand{\tocsection}[2]{\bf\oldtocsection{#1}{#2}}

\title[Non-kissing and non-crossing complexes for locally gentle algebras]{Non-kissing and non-crossing complexes \\ for locally gentle algebras}

\thanks{YP, VP and PGP were partially supported by the French ANR grant SC3A~(15\,CE40\,0004\,01). VP was partially supported by the French ANR grant CAPPS~(17\,CE40\,0018).}

\author{Yann Palu}
\address[Yann Palu]{LAMFA, Universit\'e Picardie Jules Verne, Amiens}
\email{yann.palu@u-picardie.fr}
\urladdr{\url{http://www.lamfa.u-picardie.fr/palu/}}

\author{Vincent Pilaud}
\address[Vincent Pilaud]{CNRS \& LIX, \'Ecole Polytechnique, Palaiseau}
\email{vincent.pilaud@lix.polytechnique.fr}
\urladdr{\url{http://www.lix.polytechnique.fr/~pilaud/}}

\author{Pierre-Guy Plamondon}
\address[Pierre-Guy Plamondon]{Laboratoire de Math\'ematiques d'Orsay, Universit\'e Paris-Sud, CNRS, Universit\'e Paris-Saclay}
\email{pierre-guy.plamondon@math.u-psud.fr}
\urladdr{\url{https://www.math.u-psud.fr/~plamondon/}}

\begin{document}

\begin{abstract}
Starting from a locally gentle bound quiver, we define on the one hand a simplicial complex, called the non-kissing complex.
On the other hand, we construct a punctured, marked, oriented surface with boundary, endowed with a pair of dual dissections.
From those geometric data, we define two simplicial complexes: the accordion complex, and the slalom complex, generalizing work of A. Garver and T. McConville in the case of a disk.
We show that all three simplicial complexes are isomorphic, and that they are pure and thin.
In particular, there is a notion of mutation on their facets, akin to $\tau$-tilting mutation.
Along the way, we also construct inverse bijections between the set of isomorphism classes of locally gentle bound quivers and the set of homeomorphism classes of punctured, marked, oriented surfaces with boundary, endowed with a pair of dual dissections.
\end{abstract}

\maketitle

\section{Introduction}
The aim of this paper is to prove that two combinatorial objects, called the \emph{non-kissing complex} and the \emph{non-crossing complex}, are isomorphic.
Both complexes appeared in different works in specific cases: the non-kissing complex of a grid appeared in \cite{McConville}, while the non-crossing complex of a disk appeared in \cite{GarverMcConville, MannevillePilaud-accordion}.
It was shown in \cite{PaluPilaudPlamondon} that these complexes are special cases of a more general simplicial complex, defined for any gentle algebra (see also~\cite{BrustleDouvilleMousavandThomasYildirim}).

In order to unify these two objects, we are lead to introduce two generalizations.
On the algebraic side, the non-kissing complex is extended from the class of gentle algebras to that of locally gentle algebras (which are an infinite-dimensional version of gentle algebras).
On the geometric side, we construct the non-crossing complex of an arbitrary oriented punctured surface endowed with a pair of dual dissections.

The two main objects of our study will thus be locally gentle algebras and dissections of surfaces.
Our first main result is that these two classes of objects are essentially the same.

\begin{theorem*}[\ref{thm:bijectionLocallyGentleAndSurfaces}]
There is an explicit bijection between the set of isomorphism classes of locally gentle bound quivers and the set of homeomorphism classes of oriented punctured marked surfaces with boundary endowed with a pair of dual cellular dissections.
\end{theorem*}

On the algebraic side, we extend the study of walks from the gentle case~\cite{McConville, PaluPilaudPlamondon} to the locally gentle case, and define a notion of compatibility called non-kissing (see Section \ref{sec:nonKissingComplex}).
On the geometric side, we extend the study of dissections, accordions and slaloms from the case of the disk~\cite{GarverMcConville, MannevillePilaud-accordion} to the case of an arbitrary surface, and define a notion of compatibility called non-crossing (see Section \ref{sec:accordionSlalomNonCrossingComplexes}).
The combinatorial information contained in these notions is encoded in simplicial complexes: the non-kissing complex~$\NKC$ and the non-crossing complex~$\NCC$.

\begin{theorem*}[\ref{thm:complexesCoincide}]
The complexes~$\NKC$ and~$\NCC$ are isomorphic.
\end{theorem*}

Finally, we show in Section \ref{sec:propertiesOfComplexes} that these complexes are combinatorially very well-behaved.

\begin{theorem*}[\ref{prop:purity} and \ref{coro:thin}]
The complexes~$\NKC$ and~$\NCC$ are pure and thin.  Their dimension is computed, and mutation is explicitly described.
\end{theorem*}

We end in Section \ref{sec:vectors} with a discussion of~$\b{g}$-vectors,~$\b{c}$-vectors and~$\b{d}$-vectors associated to walks and curves.
Consequences on the representation theory of locally gentle algebras will be investigated in a future project.

Finally, let us briefly review the algebraic and geometric objects which appear in this paper.

Gentle algebras are a class of finite-dimensional associative algebras over a field defined by generators and relations.
Their representation theory was first systematically investigated in~\cite{ButlerRingel}, and they have been thouroughly studied since.
Locally gentle algebras are obtained by dropping the requirement that gentle algebras be finite-dimensional.
It turns out that their representations theory is also well-behaved \cite{Crawley-Boevey} and that these algebras are Koszul \cite{BessenrodtHolm}.
Recently, the $\tau$-tilting theory \cite{AdachiIyamaReiten} of gentle algebra has been studied in \cite{PaluPilaudPlamondon, BrustleDouvilleMousavandThomasYildirim}.

Dissections of surfaces, on the other hand, are certain collections of pairwise non-intersecting curves on an orientable surface.
They have been defined and studied in \cite{Baryshnikov, Chapoton-quadrangulations, GarverMcConville, MannevillePilaud-accordion} in the case where the surface is an unpunctured disk.

The idea of associating a finite-dimensional algebra to a dissection (or a triangulation) of a surface seems to take its roots in the theory of cluster algebras and cluster categories.
This was first done for triangulations of polygons in \cite{CalderoChapotonSchiffler}, and then for any orientable surface with boundary in \cite{ABCP, Labardini}.  
The algebra of a dissection as we shall use it in this paper has appeared in \cite{DavidRoeslerSchiffler}.
In most of the above cases, the algebras obtained are gentle algebras.
It has also been shown in \cite{BaurCoelhoSimoes} that any gentle algebra is obtained from a dissection of a surface, and that the module category of the algebra can be interpreted by using curves on the surface.
In the case where the surface is a polygon, the $\tau$-tilting theory of the algebra of a dissection has been studied in \cite{PaluPilaudPlamondon, PilaudPlamondonStella}.

Conversely, the construction of a surface associated to a gentle algebra has appeared in \cite{OpperPlamondonSchroll}.
We give a different construction of the same surface in this paper, which is obtained by ``glueing'' quadrilaterals to the arrows of the blossoming quiver (as defined in \cite{PaluPilaudPlamondon}, and called ``framed quiver'' in \cite{BrustleDouvilleMousavandThomasYildirim}).
Our construction has the advantage that it easily yields the two dual dissections of the surface at the same time (the dissection and dual lamination of \cite{OpperPlamondonSchroll}).
Note that our dissections are always cellular, while those in \cite{BaurCoelhoSimoes} can be arbitrary.
Remarkably, gentle algebras and surfaces were linked recently in \cite{HaidenKatzarkovKontsevich, LekiliPolishchuk}, where the Fukaya category of the surface is shown to be equivalent to the bounded derived category of the associated gentle~algebra.

\section*{Acknowledgements}
We thank Sibylle Schroll for discussions on geometric models for gentle algebras, Salvatore Stella for discussions on accordion complexes of dissections of surfaces, and William Crawley-Boevey and Henning Krause for discussions on the representation theory of infinite dimensional algebras.

\section{Non-kissing complex}\label{sec:nonKissingComplex}

\enlargethispage{.2cm}
We first quickly review the definition of the non-kissing complex of a gentle bound quiver following the presentation of~\cite{PaluPilaudPlamondon} and extend it to locally gentle bound quivers.

\subsection{Locally gentle bound quivers and their blossoming quivers}

We consider a \defn{bound quiver}~$\bar Q = (Q,I)$, formed by a finite quiver~$Q$ and an ideal~$I$ of the path algebra~$kQ$ (the~$k$-vector space generated by all paths in~$Q$, including vertices as paths of length zero, with multiplication induced by concatenation of paths) such that~$I$ is generated by linear combinations of paths of length at least two.
Note that we \textbf{do not require} that the quotient algebra~$kQ/I$ be finite dimensional.
The following definition is adapted from~\cite{ButlerRingel}.

\begin{definition}
\label{def:gentleQuiver}
A \defn{locally gentle bound quiver}~$\bar Q \eqdef (Q,I)$ is a (finite) bound quiver~where
\begin{enumerate}[(i)]
\item each vertex~$a \in Q_0$ has at most two incoming and two outgoing arrows,
\item the ideal~$I$ is generated by paths of length exactly two,
\item for any arrow~$\beta \in Q_1$, there is at most one arrow~$\alpha \in Q_1$ such that~$t(\alpha) = s(\beta)$ and~${\alpha\beta\notin I}$ (resp.~$\alpha\beta \in I$) and at most one arrow~$\gamma \in Q_1$ such that~$t(\beta) = s(\gamma)$~and~${\beta\gamma\notin I}$~(resp.~${\beta\gamma \in I}$).
\end{enumerate}
The algebra~$kQ/I$ is called a \defn{locally gentle algebra}.
A \defn{gentle bound quiver} is a locally gentle bound quiver~$\bar Q$ such that the algebra $kQ/I$ is finite-dimensional; the algebra is then called a \defn{gentle algebra}.
\end{definition}

\begin{definition}
\label{def:blossomingQuiver}
A locally gentle bound quiver~$\bar Q$ is \defn{complete} if any vertex~$a \in Q_0$ is incident to either one ($a$ is a leaf) or four arrows ($a$ is an internal vertex).
The \defn{pruned subquiver} of a quiver~$\bar Q$ is the locally gentle quiver obtained by deleting all leaves of~$\bar Q$ (degree one vertices) and their incident arrows.
The \defn{blossoming quiver} of a locally gentle bound quiver~$\bar Q$ is the complete locally gentle bound quiver~$\bar Q\blossom$ whose pruned subquiver is~$\bar Q$.
The vertices of~$Q\blossom_0 \ssm Q_0$ are called \defn{blossom vertices}, and the arrows in~$Q\blossom_1 \ssm Q_1$ are called \defn{blossom arrows}.
\end{definition}

In other words, $\bar Q\blossom$ is obtained from~$\bar Q$ by adding blossom arrow and blossom vertices at each incomplete vertex of~$\bar Q$ and by completing~$I$ accordingly.
Examples are represented on \fref{fig:quivers}.

\begin{figure}[t]
	\capstart
	\centerline{\includegraphics[scale=.6]{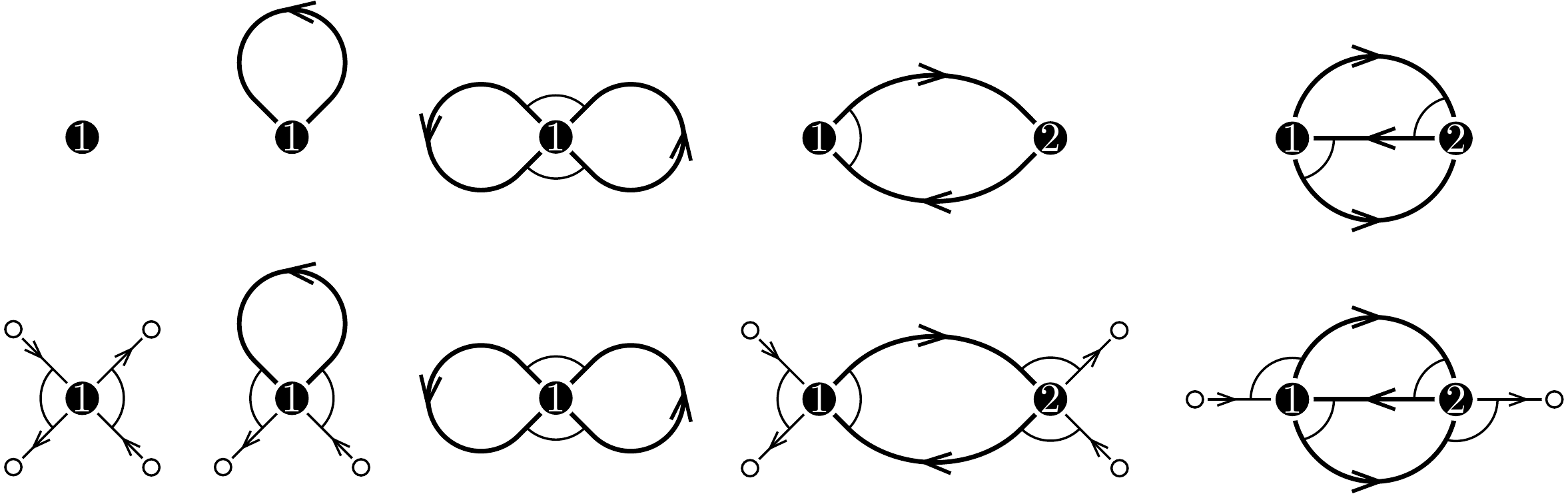}}
	\caption{Some locally gentle quivers (top) and their blossoming quivers (bottom). Initial vertices are solid and labeled while blossom vertices are hollow, and initial arrows are bold while blossom arrows are thin.}
	\label{fig:quivers}
\end{figure}

\begin{remark}
\label{rem:sizeBlossomingQuiver}
Note that~$\bar Q\blossom$ has~$2|Q_0|-|Q_1|$ incoming blossom arrows and~$2|Q_0|-|Q_1|$ outgoing blossom arrows.
Therefore, it has $|Q_0\blossom| = |Q_0| + 2(2|Q_0|-|Q_1|) = 5|Q_0|-2|Q_1|$ vertices and~${|Q_1\blossom| = |Q_1| + 2(2|Q_0|-|Q_1|) = 4|Q_0|-|Q_1|}$ arrows.
\end{remark}

\subsection{Strings and walks}

The non-kissing complex is constructed using the combinatorics of strings and walks in the quiver~$\bar Q$, whose definitions are now briefly recalled.
The terminology and notations in the following definitions is borrowed from~\cite{ButlerRingel, Crawley-Boevey}.

For any arrow~$\alpha$ of~$Q$, define a formal inverse~$\alpha^{-1}$ with the properties that~${s(\alpha^{-1}) = t(\alpha)}$, $t(\alpha^{-1}) = s(\alpha)$, $\alpha^{-1}\alpha = \varepsilon_{t(\alpha)}$ and~$\alpha\alpha^{-1} = \varepsilon_{s(\alpha)}$, where~$\varepsilon_v$ is the path of length zero starting and ending at the vertex~$v \in Q_0$.

\begin{definition}
\label{def:string}
Let~$\bar Q = (Q,I)$ be a locally gentle bound quiver.
A \defn{finite string} in~$\bar Q$ is a word of the form
\(
\rho = \alpha_1^{\varepsilon_1}\alpha_2^{\varepsilon_2}\cdots \alpha_\ell^{\varepsilon_\ell},
\)
where
	\begin{enumerate}[(i)]
	\item $\alpha_i \in Q_1$ and~$\varepsilon_i \in \{-1,1\}$ for all~$i \in [\ell]$,
	\item $t(\alpha_i^{\varepsilon_i}) = s(\alpha_{i+1}^{\varepsilon_{i+1}})$ for all~$i \in [\ell-1]$,
	\item there is no path~$\pi \in I$ such that~$\pi$ or~$\pi^{-1}$ appears as a factor of~$\rho$, and
	\item $\rho$ is reduced, in the sense that no factor~$\alpha\alpha^{-1}$ or~$\alpha^{-1}\alpha$ appears in~$\rho$, for~$\alpha \in Q_1$.
	\end{enumerate}
The integer~$\ell$ is called the \defn{length} of the string~$\rho$.
We let~$s(\rho) \eqdef s(\alpha_1^{\varepsilon_1})$ and~$t(\rho) \eqdef t(\alpha_\ell^{\varepsilon_\ell})$ denote the source and target of~$\rho$.
For each vertex~$a \in Q_0$, there is also a \defn{string of length zero}, denoted by~$\varepsilon_a$, that starts and ends at~$a$.
\end{definition}

\begin{definition}
An oriented cycle~$c$ in~$Q$ such that~$c, c^2 \notin I$ is called \defn{primitive} if it cannot be written as an $n$-th power ($n>1$) of a cycle.
\end{definition}

\begin{notation}
If~$c$ is an oriented cycle in~$Q$ such that~$c, c^2 \notin I$, we write:
\begin{itemize}
\item $c^\infty$ for the infinite word~$c c c\cdots$,
\item ${}^\infty c$ for the infinite word~$\cdots c c c$,
\item $^\infty c^\infty$ for the bi-infinite word~$^\infty c c^\infty$.
\end{itemize}
We also define~${}^{-\infty}c \eqdef {}^{\infty}(c^{-1}) = (c^\infty)^{-1}$ and similarly~$c^{-\infty} \eqdef (c^{-1})^\infty = (^\infty c)^{-1}$.
\end{notation}

\begin{definition}
An \defn{eventually cyclic string} for~$\bar Q$ is a word~$\rho$ of the form~${}^\infty(c_1^{\varepsilon_1})\sigma (c_2^{\varepsilon_2})^\infty$, where $c_1, c_2$ are oriented cycles in~$\bar Q$ (possibly of length zero) and $\varepsilon_1, \varepsilon_2$ are signs in~$\{\pm 1\}$ such that $c_1^{2\varepsilon_1} \sigma c_2^{2\varepsilon_2}$ is a finite string in~$\bar Q$.
\end{definition}

\begin{definition}
A \defn{string} for~$\bar Q$ is a word which is either a finite string or an eventually cyclic infinite string.
We often implicitly identify the two inverse strings~$\rho$ and~$\rho^{-1}$, and call it an \defn{undirected string} of~$\bar Q$.
\end{definition}

\begin{notation}
To avoid distinguishing between finite, infinite and bi-infinite words, we denote strings as products~$\rho = \prod_{i < \ell < j} \alpha_\ell$ where~$i < j \in \Z \cup \{\pm \infty\}$ and~$\alpha_\ell \in Q_0$ for all~$i < \ell < j$.
\end{notation}

\begin{definition}
\label{def:walk}
A \defn{walk} of a locally gentle quiver~$\bar Q$ is a maximal string of its blossoming quiver~$\bar Q \blossom$ (meaning that at each end it either reaches a blossom vertex of~$\bar Q \blossom$ or enters an infinite oriented cycle).
As for strings, we implicitly identify the two inverse walks~$\omega$ and~$\omega^{-1}$, and call it an \defn{undirected walk} of~$\bar Q$.
\end{definition}

\begin{definition}
\label{def:substrings}
A \defn{substring} of a walk~$\omega = \prod_{i < \ell < j} \alpha_\ell^{\varepsilon_\ell}$ of~$\bar Q$ is a string~$\sigma = \prod_{i' < \ell < j'} \alpha_\ell^{\varepsilon_\ell}$ of~$\bar Q$ for some~$i \le i' < j' \le j$, where the inequality~$i \le i'$ (resp.~$j' \le j$) is strict when~$i \ne -\infty$ (resp.~$j \ne \infty$). In other words, $\sigma$ is a factor of~$\omega$ such that
\begin{itemize}
\item the endpoints of~$\sigma$ are not allowed to be the possible blossom endpoints of~$\omega$,
\item the position of~$\sigma$ as a factor of~$\omega$ matters (the same string at a different position is considered a different substring).
\end{itemize}
Note that the string~$\varepsilon_a$ is a substring of~$\omega$ for each occurence of~$a$ as a vertex of~$\omega$.
We denote by~$\Sigma(\omega)$ the set of substrings of~$\omega$.
We use the same notation for undirected walks (of course, substrings of an undirected walk are undirected).
\end{definition}

\begin{definition}
\label{def:topBottom}
We say that the substring~$\sigma = \prod_{i' < \ell < j'} \alpha_\ell^{\varepsilon_\ell}$ is \defn{at the bottom} (resp.~\defn{on top}) of the walk~$\omega = \prod_{i < \ell < j} \alpha_\ell^{\varepsilon_\ell}$ if~$i'=-\infty$ or~$\varepsilon_{i'} = 1$ and~$j'=+\infty$ or~$\varepsilon_{j'} = -1$ (resp.~if~$i'=-\infty$ or~$\varepsilon_{i'} = -1$ and~$j'=\infty$ or~$\varepsilon_{j'} = 1$).
In other words the (at most) two arrows of~$\omega$ incident to the endpoints of~$\sigma$ point towards~$\sigma$ (resp.~outwards from~$\sigma$).
We denote by~$\Sigma_\bottom(\omega)$ and~$\Sigma_\top(\omega)$ the sets of bottom and top substrings of~$\omega$ respectively.
We use the same notations for undirected~walks.
\end{definition}

\begin{definition}
\label{def:straightBending}
A \defn{peak} (resp.~\defn{deep}) of a walk~$\omega$ is a substring of~$\omega$ of length zero which is on top (resp.~at the bottom of~$\omega$).
A \defn{corner} is either a peak or a deep (in other words, it is a vertex of~$\omega$ where the arrows change direction).
A walk~$\omega$ is \defn{straight} if it has no corner (\ie if~$\omega$ or~$\omega^{-1}$ is a path in~$\bar Q\blossom$), and \defn{bending} otherwise.
A \defn{peak walk} (resp.~\defn{deep walk}) is a walk with a unique corner (in other words, it switches orientation only once), which is a peak (resp.~deep). For~$a \in Q_0$, we denote by~$a_\peak$ the peak walk with peak at~$a$ and by~$a_\deep$ the deep walk with deep~at~$a$.
\end{definition}

\subsection{Non-kissing complex}

We can now define the non-kissing complex of a locally gentle bound quiver following~\cite{McConville, PaluPilaudPlamondon, BrustleDouvilleMousavandThomasYildirim}.
We start with the kissing relation for walks.

\begin{definition}
\label{def:kiss}
Let~$\omega$ and~$\omega'$ be two undirected walks on~$\bar Q$.
We say that~$\omega$ \defn{kisses}~$\omega'$ if ${\Sigma_\top(\omega) \cap \Sigma_\bottom(\omega')}$ contains a \emph{finite} substring.
See \fref{fig:kissing}.
We say that~$\omega$ and~$\omega'$ are \defn{kissing} if~$\omega$ kisses~$\omega'$ or~$\omega'$ kisses~$\omega$ (or both).
Note that we authorize the situation where the commmon finite substring is reduced to a vertex~$a$, meaning that~$a$ is a peak of~$\omega$ and a deep of~$\omega'$.
Observe also that~$\omega$ can kiss~$\omega'$ several times, that~$\omega$ and~$\omega'$ can mutually kiss, and that~$\omega$ can kiss itself.
\begin{figure}[h]
	\capstart
	\centerline{\includegraphics[scale=1]{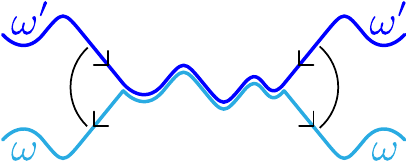}}
	\caption{A schematic representation of two kissing walks: $\omega$~kisses~$\omega'$.}
	\label{fig:kissing}
\end{figure}
\end{definition}

\begin{definition}
\label{def:nKc}
The \defn{non-kissing complex} of~$\bar Q$ is the simplicial complex~$\NKC$ whose faces are the collections of pairwise non-kissing walks of~$\bar Q$.
Note that self-kissing walks never appear in~$\NKC$ by definition.
In contrast, no straight walk can kiss another walk by definition, so that they appear in all facets of~$\NKC$.
The \defn{reduced non-kissing complex}~$\RNKC$ is the simplicial complex whose faces are the collections of pairwise non-kissing bending walks of~$\bar Q$.
\end{definition}

\section{Accordion complex, slalom complex, and non-crossing complex}\label{sec:accordionSlalomNonCrossingComplexes}

\enlargethispage{.1cm}
We now temporarily switch topic and present the accordion complex and the slalom complex of a pair of dual dissections of an orientable surface.
The latter is directly inspired from the case of a disk treated in~\cite{GarverMcConville}, while the former is obtained by duality as was already observed in the case of a disk in~\cite{MannevillePilaud-accordion}.
In the general case of arbitrary orientable surfaces, the correspondence between accordions and slaloms requires a little more attention and is treated in Section~\ref{subsec:accordionsVSSlaloms}.

\subsection{Dual dissections of a surface}

Before defining accordions and slaloms, we need a strong notion of pairs of dual dissections of an orientable surface.
We first review classical definitions of curves, arcs and dissections of a surface adapting it to our setting.

\begin{definition}
A \defn{marked surface}~$\bar\surface \eqdef (\surface, M)$ is an orientable surface~$\surface$ with boundaries, together with a set~$M$ of marked points which can be on the boundary of~$\surface$ or not.
For~$V \subset \surface$,
\begin{enumerate}[(i)]
\item a \defn{$V$-arc} on~$\bar\surface$ is a curve on~$\surface$ connecting two points of~$V$ and whose interior is disjoint from~$M$ and the boundary of~$\surface$.
\item a \defn{$V$-curve} on~$\bar\surface$ is a curve on~$\surface$ which at each end either reaches a point of~$V$ or infinitely circles around and finally reaches a puncture of~$M$, and whose interior is disjoint from~$M$ and the boundary of~$\surface$.
\end{enumerate}
As usual, curves and arcs are considered up to homotopy relative to their endpoints in~$\surface \ssm M$, and curves homotopic to a boundary are not allowed.
\end{definition}

\begin{definition}
\label{def:crossingCurves}
Two curves or arcs \defn{cross} when they intersect in~their~interior.
We will always assume that collections of arcs on a surface are in minimal position, in the sense that they cross each other transversaly, and the number of crossings is minimal.
It is pointed out in \cite{Thurston} that the results in~\cite{FreedmanHassScott} and~\cite{Neumann-Coto} imply that this assumption can always be satisfied (up to homotopy).
\end{definition}

\begin{definition}
A \defn{dissection} of~$\bar\surface$ is a collection~$\dissection$ of pairwise non-crossing arcs on~$\bar\surface$.
The \defn{edges} of~$\dissection$ are its arcs together with the boundary arcs of~$\bar\surface$.
The \defn{faces} of~$\dissection$ are the connected components of the complement of the union of the edges of~$\dissection$ in the surface~$\surface$.
We denote by~$\vertices(\dissection)$, $\edges(\dissection)$ and~$\faces(\dissection)$ the sets of vertices, edges and faces of~$\dissection$ respectively.
The dissection~$\dissection$ is \defn{cellular} if all its faces are topological disks.
For~$V \subseteq M$, a \defn{$V$-dissection} is a dissection with only $V$-arcs.
\end{definition}

\begin{convention}
\textbf{All throughout the paper, all dissections are considered cellular.}
\end{convention}

\begin{definition}
Consider a marked surface~$\bar\surface = (\surface, V \sqcup V\dual)$, where~$V$ and~$V\dual$ are two disjoint sets of marked points so that the points of~$V$ and~$V\dual$ that are on the boundary of~$\surface$ alternate.
A cellular $V$-dissection~$\dissection$ of~$\bar\surface$ and a cellular $V\dual$-dissection~$\dissection\dual$ of~$\bar\surface$ are \defn{dual cellular dissections} if there are pairs of mutually inverse bijections~$V\dual \leftrightarrow \faces(\dissection)$ and~$V \leftrightarrow \faces(\dissection\dual)$, both denoted~$\dual$ in both directions, such that~$\dissection$ has an edge joining its vertices~$u,v \in V$ and separating its faces~$f,g \in  \faces(\dissection)$ if and only if~$\dissection\dual$ has an edge joining its vertices~$f\dual, g\dual \in V\dual$ and separating its faces~$u\dual, v\dual \in \faces(\dissection\dual)$.
\end{definition}

Some examples of dual cellular dissections on different surfaces are represented in \fref{fig:dissections}.
\begin{figure}[t]
	\capstart
	\centerline{\includegraphics[scale=.7]{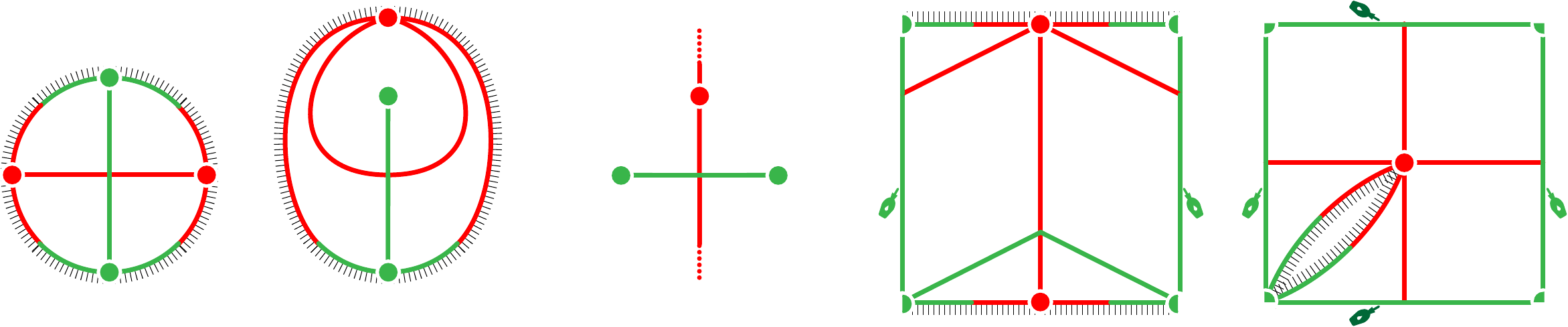}}
	\caption{Some pairs~$(\dissection, \dissection\dual)$ of dual cellular dissections on different surfaces. The dissection~$\dissection$ is in green while its dual dissection~$\dissection\dual$ is~in~red. The boundaries of the surfaces are shaded, and the glue symbols indicate how to identify some edges to get the desired surfaces. The first two surfaces are disks, the third is a sphere without boundary, the fourth is a cylinder, and the fifth is a torus with one boundary component.}
	\label{fig:dissections}
\end{figure}
Note that contrarily to the usual conventions, the dual vertex~$f\dual$ of a face~$f$ of~$\dissection$ is not always in the interior of the face~$f$.
More precisely, there are two situations:
\begin{itemize}
\item if a face~$f$ has no edge on the boundary of~$\surface$, its dual vertex~$f\dual$ lies in the interior~of~$f$ and is a puncture of~$\bar\surface$,
\item if a face~$f$ has an edge on the boundary of~$\surface$, then it has exactly one such edge and its dual vertex~$f\dual$ lies on this edge.
\end{itemize}
In fact, the second point forces the following characterization of the cellular dissections that admit a dual cellular dissection.

\begin{proposition}
\label{prop:conditionsDualDissections}
For a cellular $V$-dissection~$\dissection$ of a marked surface~$\bar\surface = (\surface, V \sqcup V\dual)$, the following assertions are equivalent:
\begin{enumerate}[(i)]
\item there exists a cellular $V\dual$-dissection~$\dissection\dual$ of~$\bar\surface$ such that~$\dissection$ and~$\dissection\dual$ are dual cellular dissections,
\item each face of~$\dissection$ contains exactly one point of~$V\dual$ (in particular, at most one boundary edge).
\end{enumerate}
Moreover, the cellular dissection~$\dissection\dual$ is uniquely determined.
\end{proposition}

\begin{proof}
We mimick the proof of~\cite[Prop.~1.12]{OpperPlamondonSchroll}.
The direct implication is immediate by definition.
Assume conversely that each face~$f$ of~$\dissection$ contains exactly one point~$f\dual$ of~$V\dual$.
We construct half-edges of~$\dissection\dual$ in each face~$f$ of~$\dissection$ by joining the dual point~$f\dual$ to the middle of each boundary edge of~$f$ that is not a boundary edge of~$\surface$.
For each edge~$a$ of~$\dissection$, the half-edges incident to the middle of~$a$ in the two faces of~$\dissection$ containing~$a$ form the dual edge~$a\dual$ of~$a$.
All these dual edges do not cross (as the half-edges do not cross in each face of~$\dissection$) and form the dual cellular dissection~$\dissection\dual$~of~$\dissection$.
\end{proof}

\begin{definition}
We consider a set~$B$ of points on the boundary of the surface~$\surface$ such that~$B$ and~$V \cup V\dual$ alternate along the boundary of~$\surface$.
The points of~$B$ are called the \defn{blossom points}.
See \eg Figures~\ref{fig:accordions} and~\ref{fig:slaloms} where the blossom points appear as white hollow vertices.
\end{definition}

\subsection{Accordion complex and slalom complex}

\enlargethispage{.1cm}
Let~$\bar\surface = (\surface, V \sqcup V\dual)$ be a surface with two disjoint sets~$V$ and~$V\dual$ of marked points, and let~$B$ be the corresponding blossom points on the boundary of~$\surface$.
We say that a $B$-curve is \defn{external} if it is homotopic to a boundary arc of~$\surface \ssm B$, and \defn{internal} otherwise.
Note that no $B$-curve can cross an external $B$-curve.
Consider two dual cellular dissections~$\dissection$ and~$\dissection\dual$ of~$\bar\surface$.

The following definition generalises the one of \cite{MannevillePilaud-accordion} for the case where~$\bar \surface$ is a disk.  A~very similar definition appears in \cite{BaurCoelhoSimoes} in a slightly different context under the \mbox{name ``permissible~arc''}.

\begin{definition}
\label{def:accordion}
A \defn{$\dissection$-accordion} is a $B$-curve~$\alpha$ of~$\bar\surface$ such that whenever~$\alpha$ meets a face~$f$~of~$\dissection$,
\begin{enumerate}[(i)]
\item it enters crossing an edge~$a$ of~$f$ and leaves crossing an edge~$b$ of~$f$ (in other words, $\alpha$ is not allowed to circle around~$f\dual$ when~$f\dual$ is a puncture),
\item the two edges~$a$ and~$b$ of~$f$ crossed by~$\alpha$ are consecutive along the boundary of~$f$,
\item $\alpha$, $a$ and~$b$ bound a disk inside~$f$ that does not contain~$f\dual{}$.
\end{enumerate}
By convention, we also consider that the punctures of~$V$ are $\dissection$-accordions that are considered external.
If we were working on the universal cover of the surface, the~$\dissection$-accordion associated to a puncture would be the infinite line crossing all (infinitely many) arcs attached to the puncture.
\end{definition}

\fref{fig:accordions} illustrates some $\dissection$-accordions for the dissections~$\dissection$ of \fref{fig:dissections}.

\begin{figure}[t]
	\capstart
	\centerline{\includegraphics[scale=.7]{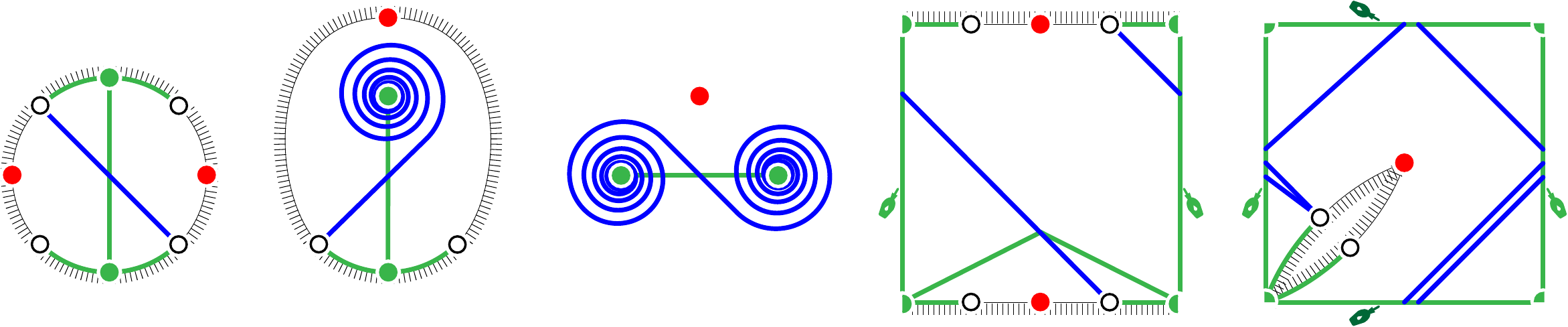}}
	\caption{Some $\dissection$-accordions (in blue) for the dissections~$\dissection$ of \fref{fig:dissections} (in green).}
	\label{fig:accordions}
\end{figure}

\begin{remark}
In Definition~\ref{def:accordion}, observe that
\begin{itemize}
\item the edges of~$a$ and~$b$ of~$f$ crossed by~$\alpha$ might coincide, see the second example in Figure~\ref{fig:accordions}.
\item the first condition is automatically satisfied if~$f\dual$ is not a puncture, \ie when it lies on the boundary of~$\surface$.
\end{itemize}
\end{remark}

\begin{definition}
\label{def:accordionComplex}
The \defn{$\dissection$-accordion complex}~$\AC$ is the simplicial complex whose faces are the collections of pairwise non-crossing $\dissection$-accordions.
Note that self-crossing accordions never appear in~$\AC$ by definition.
In contrast, no external accordion can cross another accordion, so that they appear in all facets of~$\AC$. 
The \defn{reduced $\dissection$-accordion complex}~$\RAC$ is the simplicial complex whose faces are the collections of pairwise non-crossing internal $\dissection$-accordions.
\end{definition}

\enlargethispage{.3cm}
The following definition generalizes the one from~\cite{GarverMcConville} for the case where~$\bar\surface$ is a disk.

\begin{definition}
\label{def:slalom}
A \defn{$\dissection\dual$-slalom} is a $B$-curve~$\alpha$ of~$\bar\surface$ such that, whenever~$\alpha$ crosses an edge~$a$ of~$\dissection$ contained in two faces~$f,g$ of~$\dissection$, the marked points~$f\dual$ and~$g\dual$ lie on opposite sides of~$\alpha$ in the union of~$f$ and~$g$ glued along~$a$.
Here, we consider that~$f\dual$ lies on the right (resp.~left) of~$\alpha$ when~$\alpha$ circles clockwise (resp.~counterclockwise) around~$f\dual$.
By convention, we also consider that the punctures of~$V$ are $\dissection\dual$-slaloms that are considered external.
If we were working on the universal cover of the surface, the~$\dissection$-accordion associated to a puncture would be the infinite line crossing all (infinitely many) arcs attached to the puncture.
\end{definition}

\fref{fig:slaloms} illustrates some $\dissection\dual$-slaloms for the dual dissections~$\dissection\dual$ of \fref{fig:dissections}.

\begin{figure}[t]
	\capstart
	\centerline{\includegraphics[scale=.7]{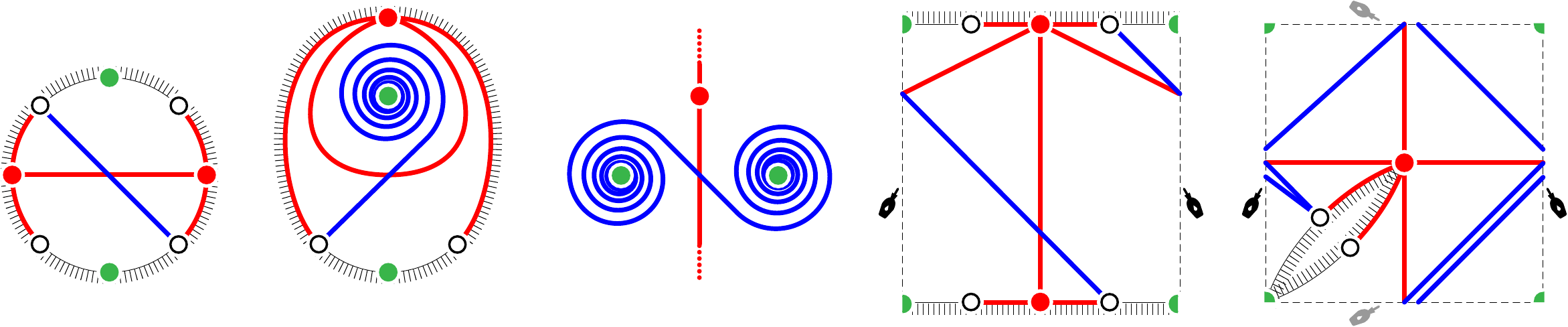}}
	\caption{Some $\dissection\dual$-slaloms (in blue) for the dual dissections~$\dissection\dual$ of \fref{fig:dissections} (in red).}
	\label{fig:slaloms}
\end{figure}

\begin{definition}
\label{def:slalomComplex}
The \defn{$\dissection\dual$-slalom complex}~$\SC$ is the simplicial complex whose faces are the collections of pairwise non-crossing $\dissection\dual$-slaloms.
Note that self-crossing slaloms never appear in~$\SC$ by definition.
In contrast, no external slalom can cross another slalom, so that they appear in all facets of~$\SC$. 
The \defn{reduced $\dissection\dual$-slalom complex}~$\RSC$ is the simplicial complex whose faces are the collections of pairwise non-crossing internal $\dissection\dual$-slaloms.
\end{definition}

\subsection{Accordions versus slaloms and the non-crossing complex}
\label{subsec:accordionsVSSlaloms}

The following statement is illustrated by Figures~\ref{fig:accordions} and~\ref{fig:slaloms}.

\begin{proposition}
\label{prop:accordionsSlaloms}
For two dual cellular dissections~$\dissection$ and~$\dissection\dual$, the $\dissection$-accordions are precisely the $\dissection\dual$-slaloms and the (reduced) $\dissection$-accordion complex coincides with the (reduced) $\dissection\dual$-slalom complex.
\end{proposition}

\begin{proof}
Let~$\alpha$ be a~$\dissection$-accordion.
Let~$a$,~$b$,~$c$ be three consecutive edges of~$\dissection$ crossed by~$\alpha$.
Let~$u \in \dissection$ be the common endpoint of~$a$ and $b$ given by Definition~\ref{def:accordion} and let~$v \in \dissection$ be the common endpoint of~$b$ and~$c$.
If $u = v$, then the segment of~$\alpha$ between~$a$ and~$c$ stays in the same face of~$\dissection\dual$.
Otherwise, when crossing~$b$, the accordion~$\alpha$ leaves the face~$u\dual$ of $\dissection\dual$ and enters the face~$v\dual$.
In the surface obtained by glueing the two cells~$u\dual$ and~$v\dual$ along~$b\dual$, the marked points~$u$ and~$v$ are separated by~$\alpha$.
Conversely, assume that~$\alpha$ is an curve of~$(\surface, B)$ which is not a $\dissection$-accordion.
We are in one of the following cases.

Case 1: The curve~$\alpha$ consecutively crosses two edges~$a$ and~$b$ of~$\dissection$ that do not share a common endpoint.
Let~$f$ be the face of~$\dissection$ that contains the segment of~$\alpha$ between~$a$ and~$b$.
Consider the connected component of~$f\setminus\alpha$ that does not contain~$f\dual$.
The boundary of this component contains an edge~$c \in \dissection \setminus \{a\}$ that shares a common endpoint with~$a$.
Let~$u,v$ be the endpoints~of~$c$.
Then~$\alpha$ cuts the cells $u\dual,v\dual$ glued along~$c\dual$ into two connected components, one containing both~endpoints~of~$\delta$.

Case 2: The curve~$\alpha$ consecutively crosses two edges~$a$ and~$b$ of~$\dissection$ that share a unique common endpoint, but the region delimited by~$\alpha$,~$a$ and~$b$ is not a disk.
Since the dissection~$\dissection$ is cellular, that region has to contain some marked point~$u \in \dissection\dual$.
One conclude that~$\alpha$ is not a~$\dissection\dual$-slalom similarly as in the first case, by considering the component of~$u\dual \setminus \alpha$ that does not contain~$u$.

Case 3: The curve~$\alpha$ consecutively crosses two edges~$a$ and~$b$ of~$\dissection$ that share two different endpoints, and none of the two regions delimited by~$\alpha$,~$a$ and~$b$ is a disk.
Since the dissection~$\dissection$ is cellular, both regions have to contain a puncture.
This contradicts the assumption that the dissections~$\dissection$ and~$\dissection\dual$ are dual to each other.

Case 4: If~$\alpha$ stays inside a $\dissection$-cell, then it rolls around a~$\dissection\dual$ puncture.
There is a~$\dissection\dual$-edge~$a$ ending in that puncture. When~$\alpha$ crosses~$a$, the puncture stays on the same side of~$\alpha$, hence~$\alpha$ is not a~$\dissection\dual$-slalom.
\end{proof}

\begin{definition}
\label{def:noncrossingComplex}
We call \defn{non-crossing complex} of the pair of dual cellular dissections~$(\dissection, \dissection\dual)$ the simplicial complex~$\NCC \eqdef \AC = \SC$.
Similarly, the \defn{reduced non-crossing complex} of~$(\dissection, \dissection\dual)$ is~$\RNCC \eqdef \RAC = \RSC$.
\end{definition}

\begin{remark}
The $\dissection$-slaloms and the $\dissection\dual$-accordions are defined dually and also coincide.
\end{remark}

\begin{example}
When the dissection~$\dissection$ is a classical dissection of a polygon (with no punctures) where each cell has at most one boundary edge, the dual dissection~$\dissection\dual$ is a tree.
The non-crossing complex in this situation was treated in detail in~\cite{GarverMcConville, MannevillePilaud-accordion}.
Note that even when the surface~$\surface$ is a disk, we authorize in the present paper both~$\dissection$ and~$\dissection\dual$ to have interior faces and punctures.
\end{example}

\begin{remark}
A quick disclaimer about accordion complexes and slalom complexes of arbitrary cellular dissections.
As stated in Proposition~\ref{prop:conditionsDualDissections}, a cellular dissection~$\dissection$ with a face containing more than one boundary edge does not admit a dual cellular dissection~$\dissection\dual$.
This prevents the use of the definitions given in this section.
Although the accordion complex (resp.~slalom complex) could still be defined as in~\cite{MannevillePilaud-accordion} (resp.~\cite{GarverMcConville}) for the disk, the resulting complexes would be joins of smaller accordion complexes (resp.~slalom complexes) defined in this paper.
See \cite[Prop.~2.4]{MannevillePilaud-accordion} for a detailed statement in the case of the disk.
\end{remark}

\section{Non-kissing versus non-crossing}

In this section, we show that non-kissing and non-crossing complexes actually coincide.
Namely, any pair of dual cellular dissections~$(\dissection, \dissection\dual)$ defines a locally gentle bound quiver~$\bar Q_{\dissection}$ (see Section~\ref{subsec:D2Q}) such that the non-crossing complex~$\NCC$ is isomorphic to the non-kissing complex~$\NKC[\bar Q_{\dissection}]$.
Conversely, any locally gentle bound quiver~$\bar Q$ gives rise to a surface~$\surface_{\bar Q}$ equipped with a pair of dual cellular dissections~$\dissection_{\bar Q}$ and~${{\dissection\dual}\!\!_{\bar Q}}$ (see Section~\ref{subsec:Q2D}) such that the non-kissing complex~$\NKC$ is isomorphic to the non-crossing complex~$\NCC[\dissection_{\bar Q}, {\dissection\dual\!\!_{\bar Q}}]$.
In fact, we show that the accordion, slalom, and non-kissing complexes can all be pictured in a unique way on the surface.
Interestingly, via the correspondence between pairs of dual cellular dissections and locally gentle bound quivers, the duality between the dissections~$\dissection$ and~$\dissection\dual$ translates into the Koszul duality between the locally gentle algebras of~$\bar Q$ and~$\bar Q\koszul$.
Many examples are treated in Section~\ref{subsec:Q2D}.

\subsection{The bound quiver of a dissection}
\label{subsec:D2Q}

Let~$\dissection$ and~$\dissection\dual$ be two dual cellular dissections of a marked surface~$(\surface, V\sqcup V\dual)$.
Let $B$ be the set of blossom points.

\begin{definition}
\label{def:quiverDualDissections}
The bound quiver of the dissection~$\dissection$ is the bound quiver $\bar Q_{\dissection} = (Q_{\dissection}, I_{\dissection})$ defined as follows:
\begin{enumerate}[(i)]
\item the set of vertices of~$Q_{\dissection}$ is the set of edges of $\dissection$;
\item there is an arrow from~$a$ to~$b$ for each common endpoint~$v$ of~$a$ and~$b$ such that~$b$ comes immediately after~$a$ in the counterclockwise order around~$v$;
\item the ideal~$I_{\dissection}$ is generated by the paths of length two in~$Q_{\dissection}$ obtained by composing arrows which correspond to triples of consecutive edges in a face of~$\dissection$.
\end{enumerate}
The bound quiver of the dissection~$\dissection\dual$ is the bound quiver $\bar Q_{\dissection\dual} = (Q_{\dissection\dual}, I_{\dissection\dual})$ defined by replacing~$\dissection$ by~$\dissection\dual$ in the above.
\end{definition}

\fref{fig:quiversDissections} illustrates this construction for the dissections of \fref{fig:dissections}.

\begin{figure}[t]
	\capstart
	\centerline{\includegraphics[scale=.7]{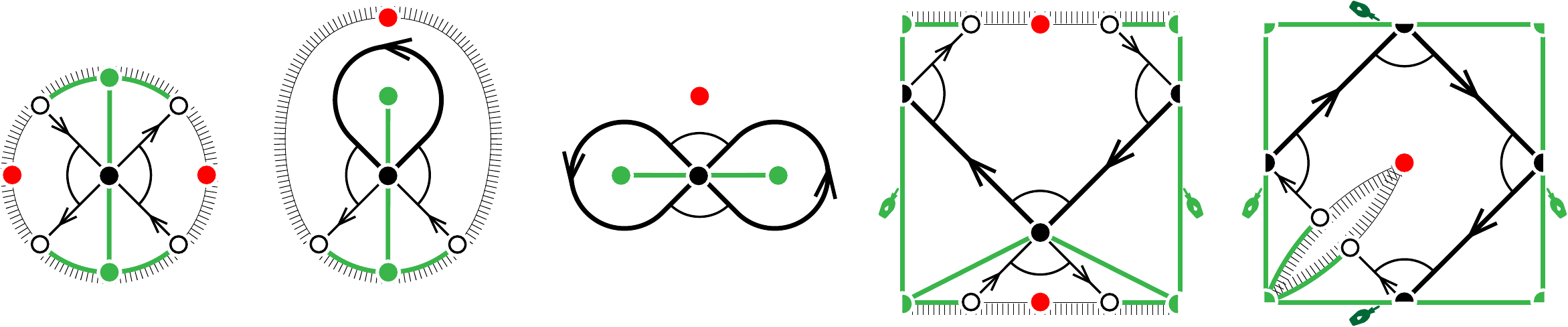}}
	\medskip
	\centerline{\includegraphics[scale=.7]{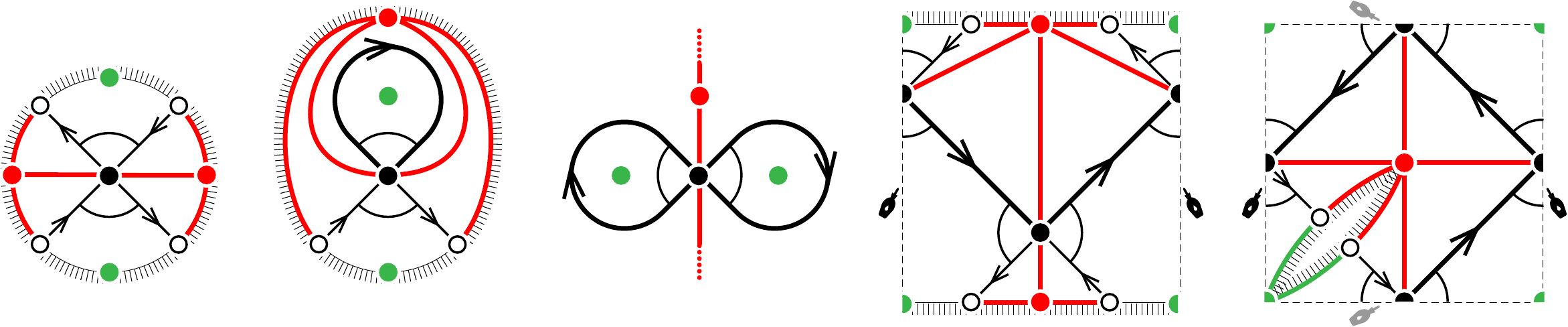}}
	\caption{The quivers associated to the dissections of \fref{fig:dissections}. As stated in Proposition~\ref{prop:dualityKoszul1}, dual dissections give rise to Koszul dual bound quivers.}
	\label{fig:quiversDissections}
\end{figure}

\begin{remark}
The blossoming quiver~$\bar Q_\dissection\blossom$ of the quiver~$\bar Q_\dissection$ is obtained with the same procedure by considering additional blossom vertices along the boundary of the surface. See \fref{fig:quiversDissections}.
\end{remark}

\begin{lemma}
\label{lemm:quiverOfDissectionIsLocallyGentle}
The bound quiver~$\bar Q_{\dissection} = (Q_{\dissection}, I_{\dissection})$ is a locally gentle bound quiver.
\end{lemma}

\begin{proof}
Each edge~$a$ of~$\dissection$ has two endpoints, and at each endpoint there is at most one edge of~$\dissection$ coming immediately after~$a$ in the clockwise order, and at most one edge of~$\dissection$ coming immediately before~$a$.
Thus, in~$Q_{\dissection}$, there are at most two arrows leaving~$a$ and at most two arriving in~$a$.
Among these arrows, those arising from the same endpoint of~$a$ will be composable, and those arising from different endpoints will yield relations.
Thus $\bar Q_{\dissection}$ is locally gentle.
\end{proof}

\begin{definition}
\label{defi:koszulDual}
The \defn{Koszul dual} of a locally gentle bound quiver~$\bar Q = (Q,I)$ is the bound quiver~$\bar Q\koszul = (Q\koszul, I\koszul)$ defined as follows:
\begin{enumerate}[(i)]
\item the quiver~$Q\koszul$ is equal to the opposite quiver of~$Q$, that is, the quiver obtained from~$Q$ by reversing all arrows;
\item the ideal~$I\koszul$ is generated by the opposites of the paths of length two in~$Q$ that do not appear in~$I$.
\end{enumerate}
\end{definition}
It is seen in \cite{BessenrodtHolm} that the (ungraded version of the) Koszul dual of a locally gentle algebra~$kQ/I$ is isomorphic to~$kQ\koszul/I\koszul$.
Note that~$\bar Q\koszul$ is locally gentle, and that Koszul duality is an involution and commutes with blossoming:~$(\bar Q\koszul)\koszul = \bar Q$ and~$(\bar Q\koszul)\blossom = (\bar Q\blossom)\koszul$.
Examples are represented on \fref{fig:koszulQuivers}.

\begin{figure}[t]
	\capstart
	\centerline{\includegraphics[scale=.6]{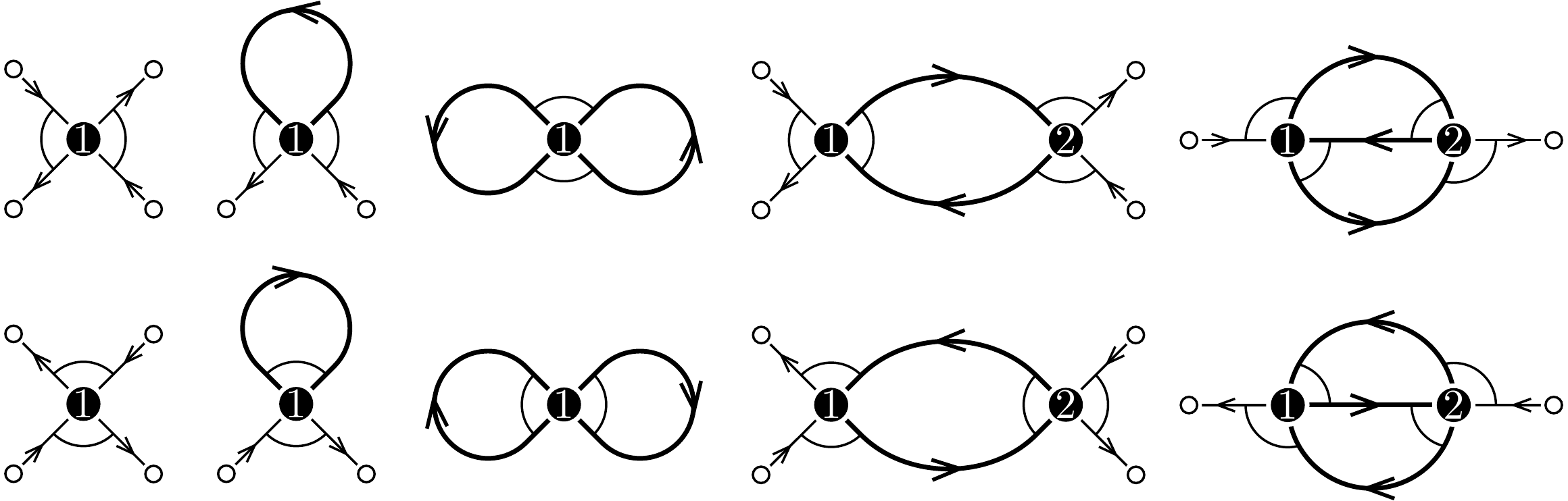}}
	\caption{The Koszul duals of the (blossoming) quivers of \fref{fig:quivers}.}
	\label{fig:koszulQuivers}
\end{figure}

The following proposition appears in \cite[Proposition 1.25]{OpperPlamondonSchroll}.  

\begin{proposition}
\label{prop:dualityKoszul1}
Let~$\dissection$ and~$\dissection\dual$ be two dual cellular dissections of a marked surface~$(\surface, V\sqcup V\dual)$.
The bound quivers~$\bar Q_{\dissection}$ and~$\bar Q_{\dissection\dual}$ are Koszul dual to each other.
\end{proposition}

\begin{proof}
The duality provides a bijection~$a\in (Q_\dissection)_0\mapsto a\dual\in (Q_{\dissection\dual})_0$, and also induces a bijection~${\alpha \in (Q_\dissection)_1 \mapsto \alpha\dual \in (Q_{\dissection\dual})_1}$ with~$s(\alpha\dual) = t(\alpha)\dual$ and~$t(\alpha\dual) = s(\alpha)\dual$.
Moreover, for any two composable arrows~$\alpha,\beta\in(Q_\dissection)_0$, we have~$\alpha\beta\in I_\dissection$ if and only if~$s(\alpha),t(\alpha),t(\beta)$ are three consecutive edges of a~$\dissection$-cell~$f$, if and only if~$s(\alpha)\dual,t(\alpha)\dual,t(\beta)\dual$ are three consecutive edges in~$\dissection\dual$ ending in~$f\dual$, if and only if~$\beta\dual\alpha\dual\notin I_{\dissection\dual}$.
\end{proof}

\subsection{The surface of a locally gentle bound quiver}
\label{subsec:Q2D}

We now associate a surface to a locally gentle bound quiver.
This construction yields the same surface as the one constructed in \cite{OpperPlamondonSchroll} (see Remark \ref{rem:comparisonWithOPS} for a comparison of the two constructions).

\begin{definition}
\label{def:surfaceQuiver}
The surface~$\surface_{\bar Q}$ of a locally gentle bound quiver~$\bar Q = (Q,I)$ is the surface obtained from the blossoming quiver~$\bar Q\blossom$ as follows:
\begin{enumerate}[(i)]
\item for each arrow~$\alpha \in Q_1\blossom$, consider a lozange~$L(\alpha)$ with four (oriented) sides
\begin{alignat*}{2}
& \Enrs{\alpha} = [v(\alpha), s(\alpha)] \qquad & \Enrt{\alpha} = [v(\alpha), t(\alpha)] \qquad & \text{(colored green)} \\
\text{and} \qquad & \Ers{\alpha} = [f\dual(\alpha), s(\alpha)] \qquad & \Ert{\alpha} = [f\dual(\alpha), t(\alpha)] \qquad & \text{(colored red)}
\end{alignat*}
placed as follows:

\centerline{
\begin{tikzpicture}[
	scale=1.5,
	thick,
	decoration={markings, mark=at position 0.5 with {\arrow{>}}},
	roundNode/.style={circle, fill=black, inner sep=2pt, outer sep=1.5pt}
	]
	\node[roundNode, label={[label distance=-3pt]left:{$s(\alpha)$}}] (s) at (-1,0) {};
	\node[roundNode, label={[label distance=-3pt]right:{$t(\alpha)$}}] (t) at (1,0) {};
	\node[roundNode, fill=green, label={[label distance=-3pt]above:{$\green v(\alpha)$}}] (v) at (0,1) {};
	\node[roundNode, fill=red, label={[label distance=-3pt]below:{$\red f\dual(\alpha)$}}] (vd) at (0,-1) {};
	\draw[color=black, postaction={decorate}] (s) -> (t) node[midway, above] {$\alpha$};
	\pic [draw, angle radius=.7cm] {angle = s--t--vd};
	\pic [draw, angle radius=.7cm] {angle = vd--s--t};
	\draw[color=green, postaction={decorate}] (v) -> (s) node[midway, above left, outer sep=-3pt] {$\green \Enrs{\alpha}$};
	\draw[color=green, postaction={decorate}] (v) -> (t) node[midway, above right, outer sep=-3pt] {$\green \Enrt{\alpha}$};
	\draw[color=red, postaction={decorate}] (vd) -> (s) node[midway, below left, outer sep=-3pt] {$\red \Ers{\alpha}$};
	\draw[color=red, postaction={decorate}] (vd) -> (t) node[midway, below right, outer sep=-3pt] {$\red \Ert{\alpha}$};
\end{tikzpicture}
}

\item for any~$\alpha, \beta \in Q_1\blossom$ with~$t(\alpha) = s(\beta)$, proceed to the following identifications:
    \begin{itemize}
    \item if~$\alpha\beta \in I$, then glue~$\Ert{\alpha}$ with~$\Ers{\beta}$,
    \item if~$\alpha\beta \notin I$, then glue~$\Enrt{\alpha}$ with~$\Enrs{\beta}$.
    \end{itemize}
\end{enumerate}
The orientations on the edges are only used to perform the identifications and can be immediately forgotten.
\end{definition}

Definition~\ref{def:surfaceQuiver} constructs an orientable surface~$\surface_{\bar Q}$ with boundaries.
\fref{fig:surfaces} illustrates this construction for the quivers of \fref{fig:quivers}.

\begin{figure}[t]
	\capstart
	\centerline{\includegraphics[scale=.7]{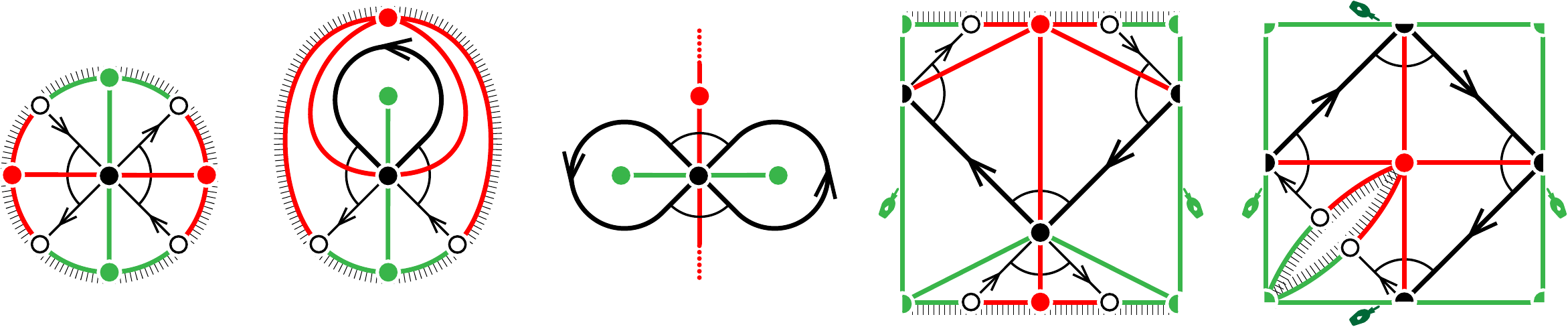}}
	\caption{The surface~$\surface_{\bar Q}$ for the quivers~$\bar Q$ of \fref{fig:quivers}.}
	\label{fig:surfaces}
\end{figure}

\begin{remark}
We could equivalently construct the surface~$\surface_{\bar Q}$ as follows:
\begin{itemize}
\item consider one $(\ell+1)$-gon~$P_\omega$ for each straight walk~$\omega$ of length~$\ell$ in~$\bar Q$ or in~$\bar Q\koszul$ (just closing~$\omega$ with an additional edge connecting~$t(\omega)$ to~$s(\omega)$),
\item for each arrow~$\alpha \in Q_1\blossom$, consider the only straight walks~$\omega$ of~$\bar Q$ and~$\omega\koszul$ of~$\bar Q\koszul$ containing~$\alpha$, and glue the polygons~$P_\omega$ and~$P_{\omega\koszul}$ along their edges corresponding to~$\alpha$.
\end{itemize}
\end{remark}

\enlargethispage{-.4cm}
The advantage of Definition~\ref{def:surfaceQuiver} is that it automatically endows~$\surface_{\bar Q}$ with two disjoint sets~$V_{\bar Q}$ and~${{V\dual}\!\!_{\bar Q}}$ of marked points and two dual cellular dissections~$\dissection_{\bar Q}$ and~${\dissection\dual\!\!_{\bar Q}}$ defined as follows.

\begin{definition}
\label{def:dissectionQuiver}
The surface~$\surface_{\bar Q}$ is endowed with
\begin{itemize}
\item the set~$V_{\bar Q}$ of points~$v(\alpha)$ for~$\alpha \in Q_1\blossom$ after the identifications given by~(ii),
\item the $V_{\bar Q}$-dissection~$\dissection_{\bar Q}$ given by all sides~$\Enrs{\alpha}$ and~$\Enrt{\alpha}$ for~$\alpha \in Q_1\blossom$ after the identifications given by~(ii).
\end{itemize}
The set~${{V\dual}\!\!_{\bar Q}}$ and the ${{V\dual}\!\!_{\bar Q}}$-dissection~${\dissection\dual\!\!_{\bar Q}}$ are defined dually.
\end{definition}

\begin{proposition}
\label{prop:dissectionsAreCellular}
Let~$\bar Q$ be a locally gentle bound quiver.
Then the dissections~$\dissection_{\bar Q}$ and~$\dissection\dual\!\!_{\bar Q}$ are cellular and dual to each other.
\end{proposition}

\begin{proof}
In the construction of~$\surface_{\bar Q}$, if two arrows~$\alpha$ and~$\beta$ of~$\bar Q \blossom$ are such that~${t(\alpha) = s(\beta)}$ and~${\alpha\beta \in I\blossom}$, then the vertices~$v(\alpha)$ and~$v(\beta)$ are identified on the surface.
Thus, any two arrows of a given path contribute the same vertex of~$V_{\bar Q}$, so~$V_{\bar Q}$ is in bijection with the set of maximal paths in~$\bar Q\blossom$.

Let~$v\in V_{\bar Q}$. The situation around~$v$ is as follows, depending on whether~$v$ corresponds to a finite or infinite maximal path of~$\bar Q\blossom$:

\begin{figure}[h]
	\capstart
	\centerline{\includegraphics[scale=.7]{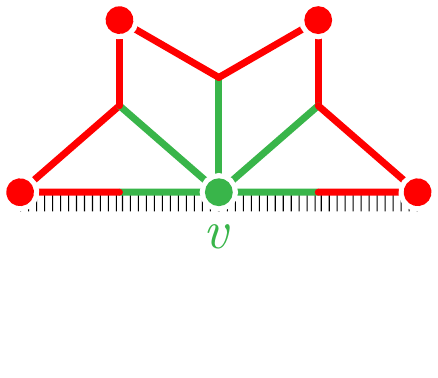} \qquad \includegraphics[scale=.7]{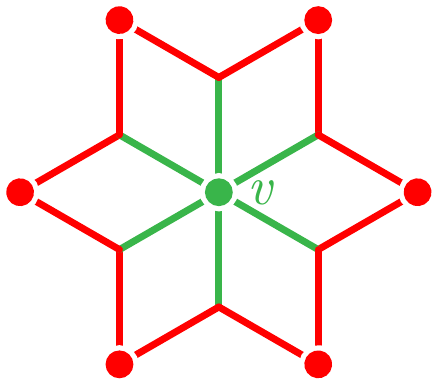}}
	\caption{Two possible situations around~$v$: on the left,~$v$ corresponds to a finite maximal path of~$\bar Q$, while on the right,~$v$ corresponds to an infinite one.}
	\label{fig:stars}
\end{figure}

In both cases, we see that~$v$ is enclosed in a polygon of~$\dissection_{\bar Q}\dual$.  This shows that~$\dissection_{\bar Q}\dual$ is a cellular dissection of~$\surface_{\bar Q}$.
Dually, one shows that~$\dissection_{\bar Q}$ is also a cellular dissection.

Moreover, by construction, each edge of~$\dissection_{\bar Q}$ which is not on the boundary of~$\surface_{\bar Q}$ crosses exactly one edge of~$\dissection_{\bar Q}\dual$, and vice versa.
This shows that~$\dissection_{\bar Q}$ and~$\dissection_{\bar Q}\dual$ are dual to each other.
\end{proof}

The following statement immediately follows from the definitions and was probably already observed by the reader on \fref{fig:surfaces}.

\begin{theorem}
\label{thm:bijectionLocallyGentleAndSurfaces}
Up to isomorphism, the constructions of Definitions \ref{def:quiverDualDissections} and \ref{def:surfaceQuiver} are inverse~to~each other.
They induce a bijection between the set of isomorphism classes of locally gentle bound quivers and the set of homeomorphism classes of marked surfaces with a pair of dual cellular dissections.
\end{theorem}

\begin{remark}
\label{rem:propertiesSurface}
The following observations are useful for the computation of examples.
\begin{enumerate}[(i)]
\item The set~$V_{\bar Q}$ has one vertex for each straight walk in~$\bar Q$ (equivalently, for each maximal path in~$\bar Q$).
      Finite straight walks yield vertices on the boundary of~$\surface_{\bar Q}$, while infinite cyclic straight walks in~$\bar Q$ yield punctures~$\surface_{\bar Q}$ in~$V_{\bar Q}$.
      We denote by~$p$ the number of infinite cyclic straight walks in~$\bar Q$.
\item \label{item:edges}
      The dissection~$\dissection_{\bar Q}$ has one edge for each vertex~$a \in \bar Q_0$, obtained by concatenation of the sides~$\Enrt{\alpha} = \Enrs{\beta}$ and~$\Enrt{\alpha'} = \Enrs{\beta'}$ where~$a = t(\alpha) = s(\beta) = t(\alpha') = s(\beta')$, $\alpha\beta \notin I$ and~$\alpha'\beta' \notin I$. We denote by~$\edgeof(a)$ the edge of~$\dissection_{\bar Q}$ corresponding to~$a$.
\item The dissection~$\dissection_{\bar Q}$ has one $\ell$-cell for each straight walk of length~$\ell$ in~$\bar Q\koszul$.
\item Similar statements hold dually for~${{V\dual}\!\!_{\bar Q}}$ and~${\dissection\dual\!\!_{\bar Q}}$, and the notations~$p\dual$ and~$\dualedgeof(a)$ are defined similarly.
\item The number of punctures of~$\surface_{\bar Q}$ is the number~$p + p\dual$ of infinite straight walks in~$\bar Q$ and in~$\bar Q\koszul$.
\item \label{item:boundary}
      The number~$b$ of boundary components of~$\surface_{\bar Q}$ can be computed as follows.
      There are two natural perfect matchings whose vertices are the blossom vertices of~$\bar Q$: one is obtained by joining the endpoints of each finite straight walk of~$\bar Q$, and the other is obtained similarly from~$\bar Q \koszul$.
      Let~$G$ be the superposition of these two perfect matchings.
      Then the number~$b$ of boundary components of~$\surface_{\bar Q}$ is the number of connected components of~$G$.
\item The genus of the surface~$\surface_{\bar Q}$ is
      \[
       g = \frac{|Q_1| - |Q_0| - b - p - p\dual + 2}{2},
      \]
      where~$b$ is the number of boundary components (see above for a way to compute~$b$) and~$p$ the number of punctures (\ie infinite straight walks in~$\bar Q$ and in~$\bar Q\koszul$).
      Indeed, ${2g = 2-\chi(\surface_{\bar Q})}$, where~$\chi(\surface_{\bar Q})$ is the Euler characteristic of~$\surface_{\bar Q}$.
      The dissection~$\dissection$ provides a cellular decomposition of~$\surface_{\bar Q}$ on which we fill the boundary components with disks.  
      The number of faces of this cellular decomposition is~$b + p\dual$, its number of edges is~$|Q_0|$ and its number of vertices is the number of straight walks in~$\bar Q$, which is equal to~$2|Q_0| - |Q_1| + p$.
      The above formula for~$g$ follows.
\end{enumerate}
\end{remark}

\begin{proposition}
\label{prop:dualityKoszul2}
For any gentle bound quiver~$\bar Q$ with Koszul dual~$\bar Q\koszul$, the surfaces~$\surface_{\bar Q}$ and~$\surface_{\bar Q\koszul}$ coincide, but~$\dissection_{\bar Q\koszul} = {\dissection\dual\!\!_{\bar Q}}$ and~${\dissection\dual\!\!_{\bar Q\koszul}} = \dissection_{\bar Q}$.
\end{proposition}

\begin{remark}
\label{rem:comparisonWithOPS}
The construction of the surface given in Definition \ref{def:surfaceQuiver} yields the same surface as the one constructed in \cite{OpperPlamondonSchroll}.
A notable difference is that in \cite{OpperPlamondonSchroll}, only~$\dissection_{\bar Q}$ is given, and~$\dissection\dual\!\!_{\bar Q}$ is deduced (it is called the dual lamination).
Another minor difference is that the construction in \cite{OpperPlamondonSchroll} is written only for gentle algebras, while here it is generalized to locally gentle algebras.

The major difference with \cite{OpperPlamondonSchroll} is the application of the surface of a (locally) gentle algebra: there, it is used as a model for the derived category of the gentle algebra (graded arcs correspond to indecomposable objects in that category), while here, it will be used as a model for walks in~$\bar Q$ ($\dissection_{\bar Q}$-accordions will correspond to walks in~$\bar{Q}$, see Proposition \ref{prop:walks=arcs}).
\end{remark}

We conclude this section with five families of examples that illustrate the computation of~$\surface_{\bar Q}$.

\begin{example}
\label{exm:CambrianPath}
Consider a \defn{Cambrian quiver}, that is any orientation of a line with no relations (\aka type~$A$ quiver).
Two such quivers are represented in \fref{fig:CambrianPathsQuivers}.
Note that we can choose to position the blossom vertices in such a way that for any arrow~$\beta$, the arrows~$\alpha$ and~$\gamma$ such that~$\alpha\beta \in I$ and~$\beta\gamma \in I$ are on the right of~$\beta$.

\begin{figure}[h]
	\capstart
	\centerline{\includegraphics[scale=.45]{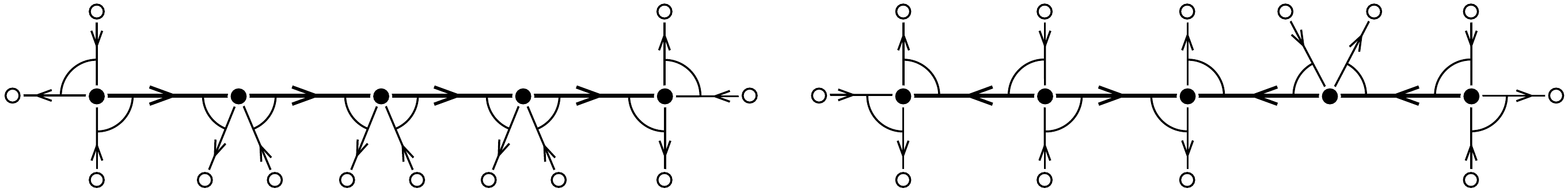}}
	\caption{Two Cambrian quivers on~$5$ vertices.}
	\label{fig:CambrianPathsQuivers}
\end{figure}

We apply Remark~\ref{rem:propertiesSurface} to understand the corresponding surfaces.
The corresponding perfect matchings (see Remark~\ref{rem:propertiesSurface}\,\eqref{item:boundary}) form a cycle with green and red arrows alternating.
Moreover, $\bar Q$ and~$\bar Q\koszul$ have no infinite straight walks.
Since~$|Q_1| = |Q_0|-1$, the corresponding surfaces are disks ($1$ boundary component, no puncture and genus~$0$).
Two examples are represented in \fref{fig:reversedPathsSurfaces}.
Note that the dissection~$\dissection_{\bar Q}$ is a triangulation with no internal triangles, and the orientation of~$\bar Q$ indicates how to glue these triangles.
For instance, the Cambrian quiver completely oriented in one direction yields a fan triangulation, were all internal edges are incident to the same vertex, as in \fref{fig:CambrianPathsSurfaces}\,(left).

\begin{figure}[h]
	\capstart
	\centerline{\includegraphics[scale=.7]{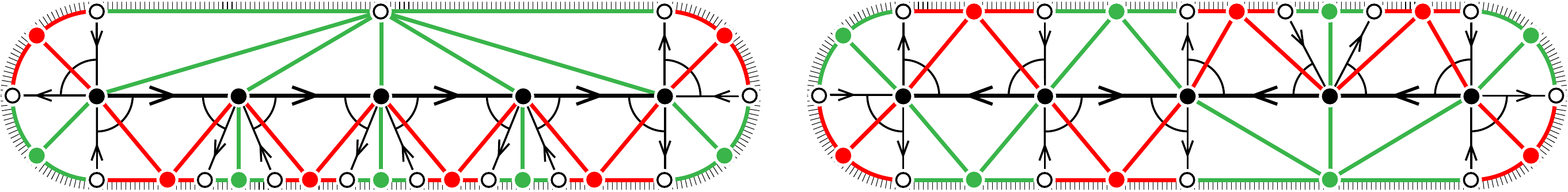}}
	\caption{The surface~$\surface_{\bar Q}$ for two Cambrian paths on~$5$ vertices.}
	\label{fig:CambrianPathsSurfaces}
\end{figure}
\end{example}

\begin{example}
\label{exm:reversedPath}
Consider the family of \defn{reversed path quivers} indicated in \fref{fig:reversedPathsQuivers}.

\begin{figure}[h]
	\capstart
	\centerline{\includegraphics[scale=.45]{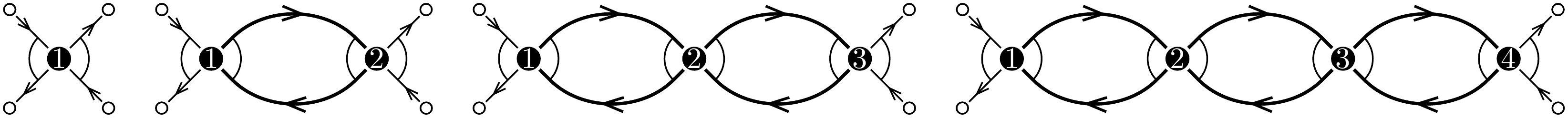}}
	\caption{The reversed path quivers with~$1, 2, 3, 4$ vertices.}
	\label{fig:reversedPathsQuivers}
\end{figure}

We apply Remark~\ref{rem:propertiesSurface} to understand the corresponding surfaces.
The corresponding perfect matchings (see Remark~\ref{rem:propertiesSurface}\,\eqref{item:boundary}) look like~\raisebox{-.1cm}{\includegraphics[scale=.3]{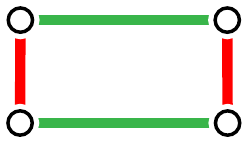}}.
Moreover, $\bar Q$ has no infinite straight walk, while~$\bar Q\koszul$ has~$|Q_0|-1$ infinite straight walks.
Since~$|Q_1| = 2|Q_0|-2$, the corresponding surfaces have $1$ boundary component and genus~$0$.
The first few corresponding surfaces are represented in \fref{fig:reversedPathsSurfaces}.

\begin{figure}[h]
	\capstart
	\centerline{\includegraphics[scale=.7]{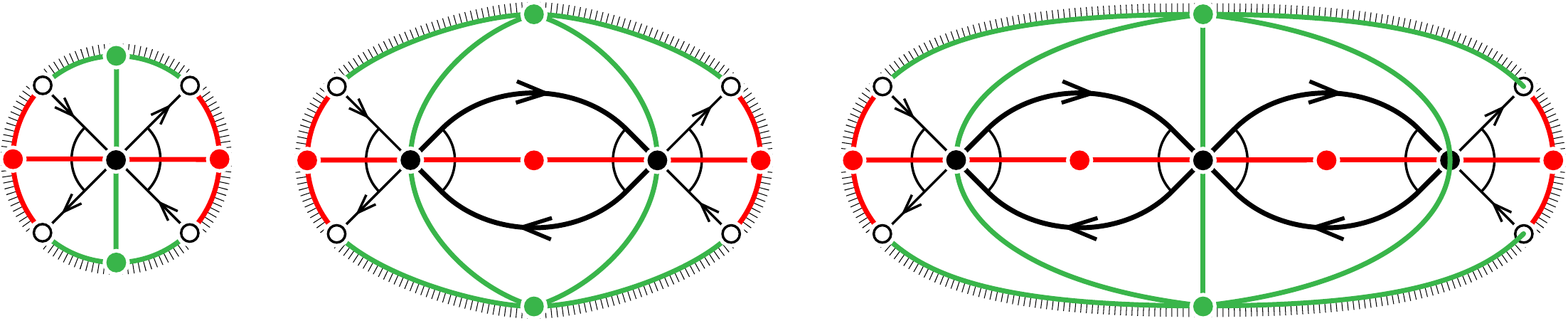}}
	\caption{The surface~$\surface_{\bar Q}$ for the reversed path quivers with~$1, 2, 3$ vertices.}
	\label{fig:reversedPathsSurfaces}
\end{figure}
\end{example}

\begin{example}
\label{exm:doublePath}
Consider the family of \defn{double path quivers} indicated in \fref{fig:doublePathsQuivers}.

\begin{figure}[h]
	\capstart
	\centerline{\includegraphics[scale=.45]{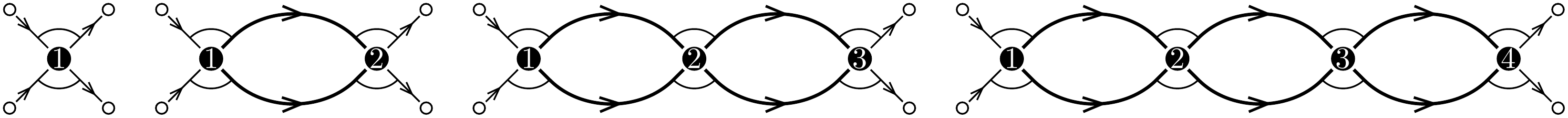}}
	\caption{The double path quivers with~$1, 2, 3, 4$ vertices.}
	\label{fig:doublePathsQuivers}
\end{figure}

We apply Remark~\ref{rem:propertiesSurface} to understand the corresponding surfaces.
The corresponding perfect matchings (see Remark~\ref{rem:propertiesSurface}\,\eqref{item:boundary}) look like~\raisebox{-.1cm}{\includegraphics[scale=.3]{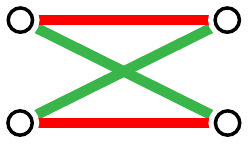}} when~$|Q_0|$ is odd and~\raisebox{-.1cm}{\includegraphics[scale=.3]{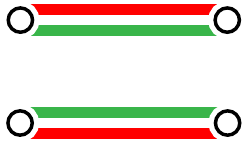}} when~$|Q_0|$ is even.
Moreover, $\bar Q$ and~$\bar Q\koszul$ have no infinite straight walks.
Since~$|Q_1| = 2|Q_0|-2$, the corresponding surfaces have no punctures and
\begin{itemize}
\item $1$ boundary component and genus~$(|Q_0|-1)/2$ when~$|Q_0|$ is odd, and 
\item $2$ boundary components and genus~$(|Q_0|-2)/2$ when~$|Q_0|$ is even.
\end{itemize}
The first few corresponding surfaces are represented in \fref{fig:doublePathsSurfaces}.

\begin{figure}[H]
	\capstart
	\centerline{\includegraphics[scale=.7]{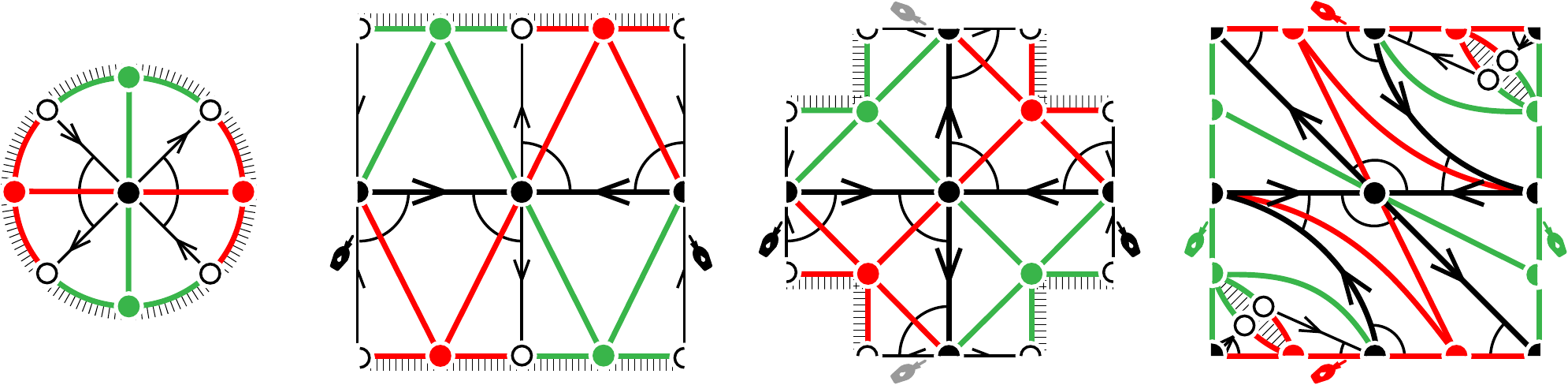}}
	\caption{The surface~$\surface_{\bar Q}$ for the double path quivers with~$1, 2, 3, 4$ vertices.}
	\label{fig:doublePathsSurfaces}
\end{figure}
\end{example}

\begin{example}
\label{exm:cycle}
Consider the family of \defn{cycle quivers} indicated in \fref{fig:cyclesQuivers}.

\begin{figure}[H]
	\capstart
	\centerline{\includegraphics[scale=.45]{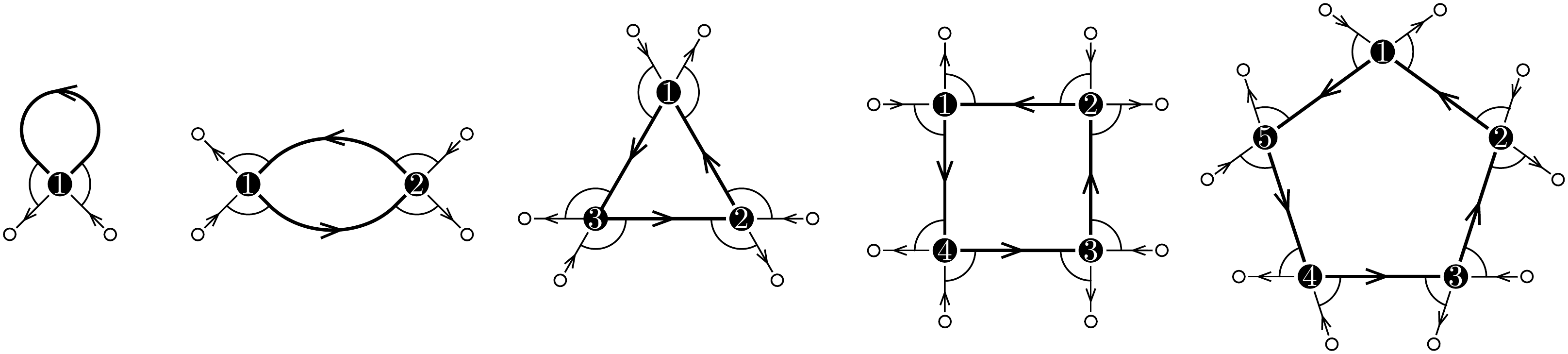}}
	\caption{The cycle quivers with~$1, 2, 3, 4, 5$ vertices.}
	\label{fig:cyclesQuivers}
\end{figure}

We apply Remark~\ref{rem:propertiesSurface} to understand the corresponding surfaces.
The corresponding perfect matchings (see Remark~\ref{rem:propertiesSurface}\,\eqref{item:boundary}) form a cycle with green and red arrows alternating.
Moreover, $\bar Q$ has~$1$ infinite straight walk, while~$\bar Q\koszul$ has none.
Since~$|Q_1| = |Q_0|$, the corresponding surfaces are punctured disks ($1$ boundary component, $1$ puncture and genus~$0$).
The first few corresponding surfaces are represented in \fref{fig:cyclesSurfaces} (see Figures~\ref{fig:surfaces} and~\ref{fig:reversedPathsSurfaces} for the cycles with $1$ or~$2$ vertices, recalling that~$\bar Q$ and~$\bar Q\koszul$ yield the same surface).

\begin{figure}[h]
	\capstart
	\centerline{\includegraphics[scale=.7]{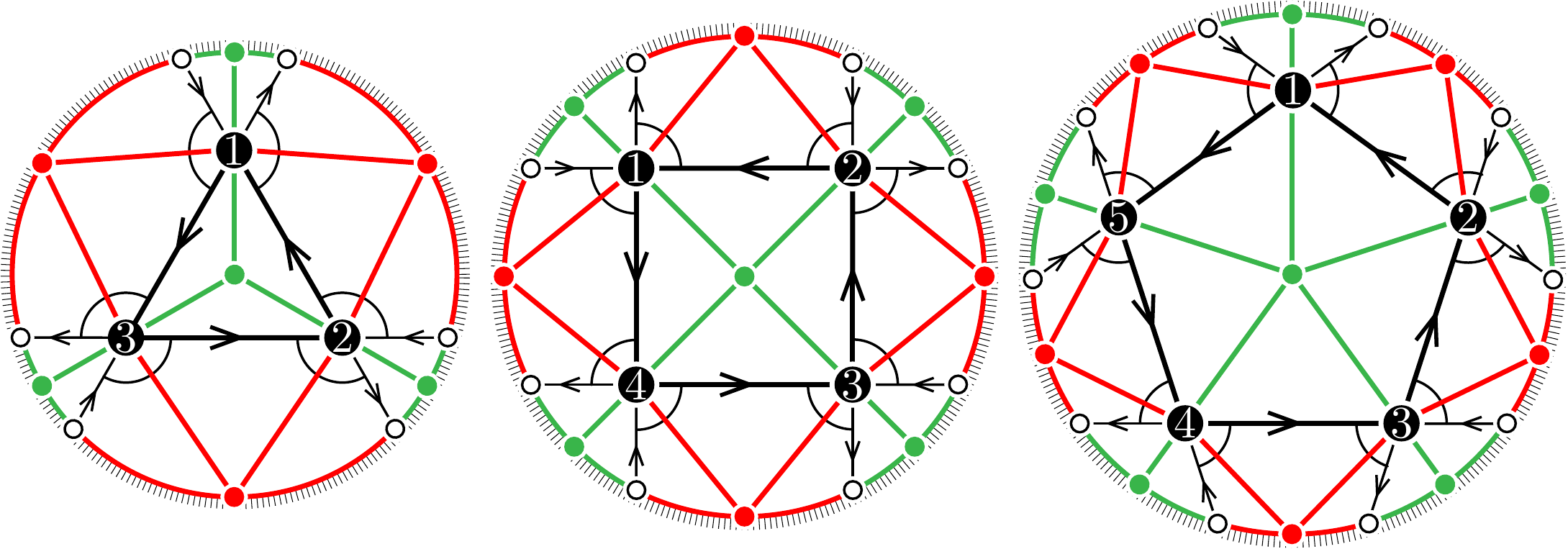}}
	\caption{The surface~$\surface_{\bar Q}$ for the cycle quivers with~$3, 4, 5$ vertices.}
	\label{fig:cyclesSurfaces}
\end{figure}
\end{example}

\begin{example}
\label{exm:doubleCycle}
Consider the family of \defn{double cycle quivers} indicated in \fref{fig:doubleCyclesQuivers}.

\begin{figure}[H]
	\capstart
	\centerline{\includegraphics[scale=.45]{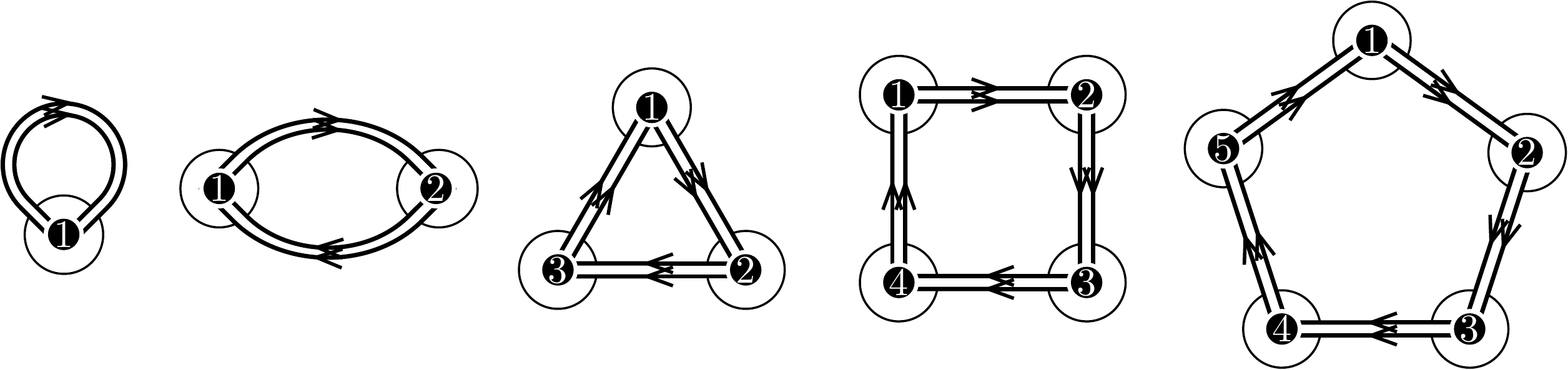}}
	\caption{The double cycle quivers with~$1, 2, 3, 4, 5$ vertices.}
	\label{fig:doubleCyclesQuivers}
\end{figure}

We apply Remark~\ref{rem:propertiesSurface} to understand the corresponding surfaces.
The corresponding perfect matchings (see Remark~\ref{rem:propertiesSurface}\,\eqref{item:boundary}) are empty since there are no blossom vertices.
Moreover, $\bar Q$ has~$1$ infinite straight walk when~$|Q_0|$ is odd and~$2$ when~$|Q_0|$ is even, and~$\bar Q\koszul$ always has~$2$ infinite straight walks.
Since~$|Q_1| = 2|Q_0|$, the corresponding surfaces have no boundary component and
\begin{itemize}
\item $3$ punctures and genus~$(|Q_0|-1)/2$ when~$|Q_0|$ is odd, and 
\item $4$ punctures and genus~$(|Q_0|-2)/2$ when~$|Q_0|$ is even.
\end{itemize}
The first few corresponding surfaces are represented in \fref{fig:doubleCyclesSurfaces}.

\begin{figure}[h]
	\capstart
	\centerline{\includegraphics[scale=.7]{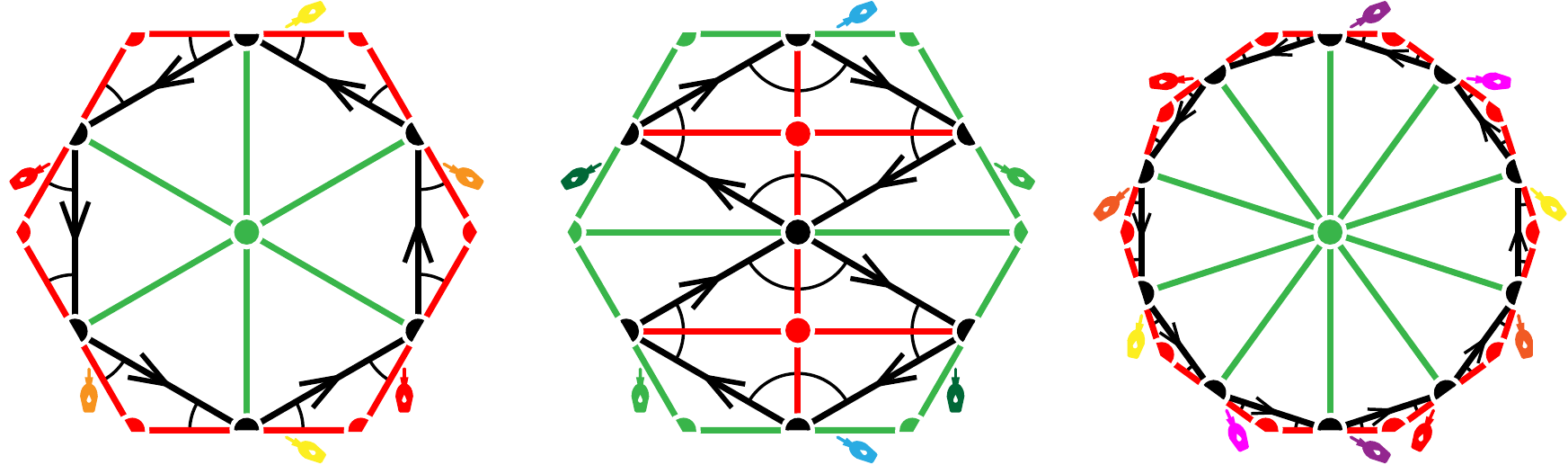}}
	\caption{The surface~$\surface_{\bar Q}$ for the double cycle quivers with~$3, 4, 5$ vertices.}
	\label{fig:doubleCyclesSurfaces}
\end{figure}
\end{example}

\subsection{Non-crossing and non-kissing complexes coincide}

Let~$\bar Q$ be a locally gentle bound quiver.
For each edge of the dissection~$\dissection_{\bar Q}$ on~$\surface_{\bar Q}$, we fix a point on the interior of this edge, which we call its ``middle point'' (this is the black vertex on the pictures).
To each walk on~$\bar Q$, we will now associate a~$\dissection_{\bar Q}$-accordion.

\begin{definition}
\label{def:curveOfAnArrow}
For any arrow~$\alpha$ on $\bar Q\blossom$, let~$\curveof(\alpha)$ be the curve on~$\surface_{\bar Q}$ which goes from the middle point of the edge of~$\dissection$ corresponding to $s(\alpha)$ to the middle point of the edge of~$\dissection$ corresponding to $t(\alpha)$ by following the angle corresponding to~$\alpha$.
Define~$\curveof(\alpha^{-1})$ to be~$\curveof(\alpha)^{-1}$.
\end{definition}

\begin{definition}
\label{def:curveOfAWalk}
Let~$\omega = \prod_{i < \ell < j} \alpha_\ell^{\varepsilon_\ell}$ be a (possibly infinite) walk on~$\bar Q$. 
Define the curve~$\curveof(\omega)$ to be the concatenation of the curves~$\curveof(\alpha_\ell^{\varepsilon_\ell})$ of Definition \ref{def:curveOfAnArrow}.
\end{definition}

In practice, we will represent~$\curveof(\omega)$ by a curve which intersects itself only transversaly, and such that if it circles infinitely around a puncture, then it spirals towards it.

\begin{lemma}
\label{lemm:curveOfAWalkIsAccordion}
Let~$\omega$ be a walk on~$\bar Q$.  Then~$\curveof(\omega)$ is a~$\dissection_{\bar Q}$-accordion.
\end{lemma}

\begin{proof}
This is because~$\curveof(\omega)$ follows angles in~$\dissection_{\bar Q}$ by definition.
\end{proof}

\begin{lemma}
\label{lemm:accordionsAreCurvesOfWalks}
Let~$\gamma$ be a~$\dissection_{\bar Q}$-accordion.  There exists a unique undirected walk~$\walk(\gamma)$ such that we have~$\curveof(\walk(\gamma)) = \gamma$.
\end{lemma}

\begin{proof}
We can assume that~$\gamma$ intersects itself and the edges of~$\dissection_{\bar Q}$ transversaly and minimally, in the sense that~$\gamma$ does not cross an edge twice in succession in opposite directions (see Definition~\ref{def:crossingCurves}).
Then~$\gamma$ is completely determined by its sequence of intersection points with the edges of~$\dissection_{\bar Q}$, since the dissection is cellular.
Two successive intersection points in this sequence define an angle between two edges of~$\dissection_{\bar Q}$, which corresponds to an arrow in~$\bar Q$.
Thus the sequence of intersection points defines a walk~$\walk(\gamma)$ on~$\bar Q$.
By construction, we have that~$\curveof(\walk(\gamma)) = \gamma$.
Is also follows from the construction above and from Definition~\ref{def:curveOfAWalk} that~$\walk(\curveof(\omega)) = \omega$ for any walk~$\omega$.
This finishes the proof.
\end{proof}

The above implies the following.

\begin{proposition}
\label{prop:walks=arcs}
The maps~$\curveof(-)$ and~$\walk(-)$ induce mutually inverse bijections between the set of undirected walks on~$\bar Q$ and the set of~$\dissection_{\bar Q}$-accordions on~$\surface_{\bar Q}$.
\end{proposition}

\begin{lemma}
\label{lem:nonKissing=nonCrossing}
Two undirected walks~$\omega_1$ and~$\omega_2$ on~$\bar Q$ are non-kissing if and only if the corresponding $\dissection_{\bar Q}$-accordions~$\curveof(\omega_1)$ and~$\curveof(\omega_2)$ are non-crossing on~$\surface_{\bar Q}$.
\end{lemma}

\begin{proof}
A kiss between two walks~$\omega$ and~$\omega'$ is depicted on \fref{fig:kissings}.

\begin{figure}[h]
	\capstart
	\centerline{\includegraphics[scale=1]{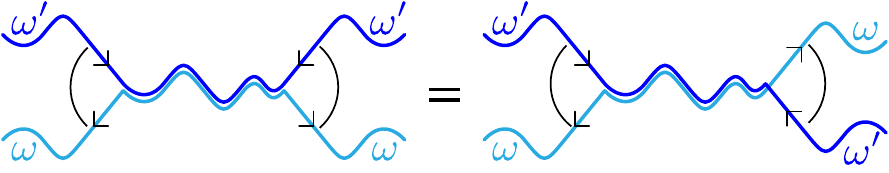}}
	\caption{Two representations of a pair of kissing walks.}
	\label{fig:kissings}
\end{figure}
The left hand side of the picture is as in \fref{fig:kissing}, and the right hand side is simply a different representation of it.
On the surface~$\surface_{\bar Q}$, the dual dissections~$\dissection$ and~$\dissection\dual$ are as on \fref{fig:kissingVSCrossing} (right).

\begin{figure}[h]
	\capstart
	\centerline{\includegraphics[scale=1]{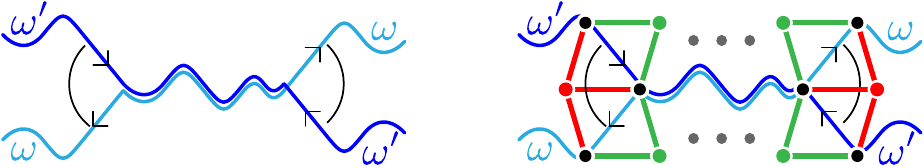}}
	\caption{Kissing walks on the quiver~$\bar Q$ correspond to crossing curves on the surface~$\surface_{\bar Q}$.}
	\label{fig:kissingVSCrossing}
\end{figure}

The curves~$\curveof(\omega)$ and~$\curveof(\omega')$ on the surface follow~$\omega$ and~$\omega'$ on the picture.
It is then clear that~$\omega$ and~$\omega'$ kiss if and only if~$\curveof(\omega)$ and~$\curveof(\omega')$ cross. 
\end{proof}

\begin{theorem}
\label{thm:complexesCoincide}
The non-kissing and non-crossing complexes are isomorphic:
\begin{itemize}
\item for any locally gentle bound quiver~$\bar Q$, the non-kissing complex~$\NKC$ is isomorphic to the non-crossing complex~$\NCC[\dissection_{\bar Q}, {\dissection\dual\!\!_{\bar Q}}]$,
\item for any pair of dual cellular dissections~$\dissection, \dissection\dual$ of an oriented surface, the non-crossing complex~$\NCC$ is isomorphic to the non-kissing complex~$\NKC[\bar Q_{\dissection, \dissection\dual}]$.
\end{itemize}
\end{theorem}

\begin{proof}
The first point is a consequence of Proposition \ref{prop:walks=arcs}, Lemma \ref{lem:nonKissing=nonCrossing} and Proposition \ref{prop:accordionsSlaloms}.
The second point is a consequence of Theorem \ref{thm:bijectionLocallyGentleAndSurfaces} and of the first point.
\end{proof}

\begin{remark}
Note that our construction of the surface~$\surface_{\bar Q}$ implies that the quiver~$\bar Q$ is naturally embedded on~$\surface_{\bar Q}$.
If we force all $\dissection$-accordions (or $\dissection\dual$-slaloms) to follow the arrows of~$\bar Q$, then a $\dissection$-accordion (or $\dissection\dual$-slalom) is really seen as a walk on~$\bar Q$.
\end{remark}

We conclude with an important example of facets of the non-kissing and non-crossing complexes.

\begin{example}
\label{exm:peakDeepFacets}
We have seen in Definition~\ref{def:straightBending} that each vertex~$a$ of~$Q$ gives rise to a \defn{peak walk}~$a_\peak$ (resp.~\defn{deep walk}~$a_\deep$) with a single peak (resp.~deep) at~$a$ and no other corner.
The set of all such walks forms the \defn{peak facet}~$\set{a_\peak}{a \in Q_0}$ (resp.~the \defn{deep facet}~$\set{a_\deep}{a \in Q_0}$) of the non-kissing complex~$\RNKC$.
As an example of Theorem~\ref{thm:complexesCoincide}, let us now describe the corresponding non-crossing facets of~$\RNCC[\dissection_{\bar Q}, {{\dissection\dual}\!\!_{\bar Q}}]$.

For an edge~$a$ of~$\dissection$, we denote by~$\vprevious{a}$ (resp.~$\vnext{a}$) the curve obtained by moving each endpoint~$v$ of~$a$ as follows:
\begin{itemize}
\item if~$v$ is on the boundary of~$\surface$, then move (continuously) $v$ until it reaches the following blossom vertex along the boundary of~$\surface$ while keeping the boundary of~$\surface$ on the right,
\item if~$v$ is a puncture, then rotate around~$v$ in counterclockwise (resp.~clockwise) direction.
\end{itemize}
Examples are illustrated in \fref{fig:peakDeep}.
\begin{figure}[h]
	\capstart
	\centerline{\includegraphics[scale=.7]{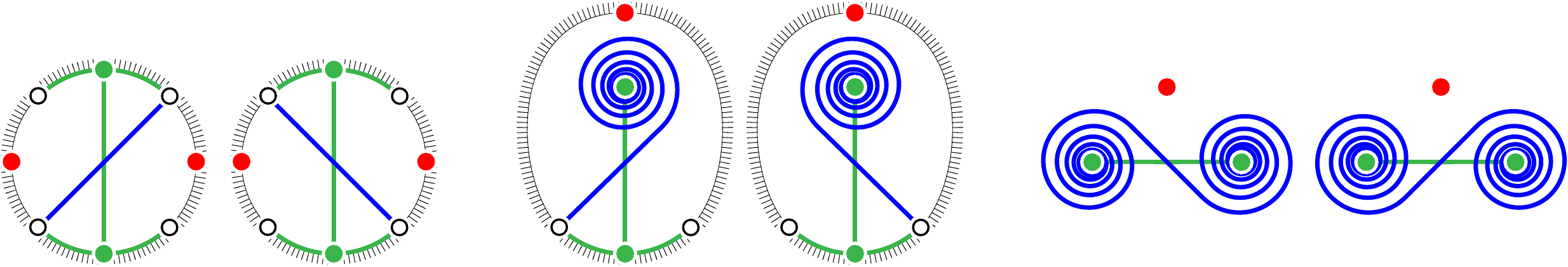}}
	\caption{The curves~$\vprevious{a}$ (left) and $\vnext{a}$ (right) associated to an edge~$a$ of~$\dissection$ for three different dissections.}
	\label{fig:peakDeep}
\end{figure}

We invite the reader to check that~$\curveof(a_\peak) = \vprevious{\edgeof(a)}$ (resp.~$\curveof(a_\deep) = \vnext{\edgeof(a)}$), where~$\edgeof(a)$ denotes the edge of~$\dissection_{\bar Q}$ corresponding to~$a \in Q_0$ (see Remark~\ref{rem:propertiesSurface}\,\eqref{item:edges}).
Therefore, the peak facet (resp.~the deep facet) can be thought of as the dissection~$\dissection_{\bar Q}$ slightly rotated clockwise (resp.~counterclockwise) on the surface~$\surface_{\bar Q}$.
\end{example}

\begin{remark}
According to Proposition~\ref{prop:dualityKoszul2} and Theorem~\ref{thm:complexesCoincide}, the simplicial complexes~$\NKC[\bar Q\koszul]$ and~$\NCC[{\dissection\dual\!\!_{\bar Q}}, \dissection_{\bar Q}]$ are isomorphic.
It would be interesting to find the precise link between~$\NKC$ and~$\NKC[\bar Q\koszul]$.
\end{remark}

\section{Properties of the non-crossing and non-kissing complexes}\label{sec:propertiesOfComplexes}

This section explores combinatorial properties of the non-crossing and non-kissing complexes, showing in particular that these complexes are pure and thin.
The proof follows ideas of~\cite{McConville, GarverMcConville} already reinterpreted for non-kissing complexes of gentle algebras in~\cite{PaluPilaudPlamondon}.
However, the setting of locally gentle algebras studied in this paper requires a little more care.

\subsection{The countercurrent order}

The countercurrent order was introduced in~\cite{McConville} for grids and defined in \cite{PaluPilaudPlamondon} for gentle algebras.
We now adapt it to locally gentle algebras.

\begin{definition}
A \defn{marked walk}~$\omega_\star$ is a walk~$\omega$ together with a marked occurence of an arrow~$\alpha^{\pm}$ in~$\omega$.
In the case where~$\omega$ is an infinite straight walk (so~$\omega$ is a cycle which repeats infinitely on both sides), we consider that all marked occurences of an arrow~$\alpha^\pm$ are equal.
Otherwise, if~$\omega$ contains several occurrences of~$\alpha^\pm$, only one occurrence is marked.
\end{definition}

The following definition is illustrated in \fref{fig:orderCurves}\,(left).

\begin{definition}
\label{def:orderForWalks}
For any arrow~$\alpha \in Q\blossom_1$ and any two distinct non-kissing walks~$\mu_\star, \nu_\star$ marked at an occurrence of~$\alpha^\pm$,
let~$\sigma$ denote their maximal common substring containing that occurrence of~$\alpha$.
Since~$\mu_\star \ne \nu_\star$, this common substring~$\sigma$ is strict, so~$\mu_\star$ and~$\nu_\star$ split at one endpoint of~$\sigma$.
The \defn{countercurrent order at~$\alpha$} is defined by~$\mu_\star \prec_\alpha \nu_\star$ when~$\mu_\star$ enters and/or exits~$\sigma$ in the direction pointed by~$\alpha$, while~$\nu_\star$ enters and/or exits~$\sigma$ in the direction opposite to~$\alpha$.
\end{definition}

\begin{remark}
Since the walks~$\mu$ and~$\nu$ are non-kissing, if~$\mu_\star$ leaves~$\sigma$ on both sides, then is enters and exits~$\sigma$ in the same direction, so that Definition~\ref{def:orderForWalks} is coherent.
\end{remark}

\begin{remark}
\label{rem:countercurrentOrderOnSurface}
Using Theorem \ref{thm:complexesCoincide}, we can translate the definition of the countercurrent order~$\prec_\alpha$ to the language of~$\dissection$-accordions (resp.~of~$\dissection\dual$-slaloms).
An arrow~$\alpha$ in~$\bar Q\blossom$ corresponds to an angle at a point~$v$ between two edges of~$\dissection$ (resp.~at a point~$f\dual$ between two edges of~$\dissection\dual$).
Two distinct non-kissing walks~$\mu_\star, \nu_\star$ marked at~$\alpha$ in~$\bar Q$ correspond to two distinct non-crossing~$\dissection$-accordions (resp.~$\dissection\dual$-slaloms)~$\gamma$ and~$\delta$ passing in front of this angle.
Then~$\gamma \prec_\alpha \delta$ if~$\gamma$ passes closer to~$v$ (resp.~further to~$f\dual$) than~$\delta$ does.
These interpretations of~$\prec_\alpha$ on the surface where already considered in the case of the disk in~\cite{GarverMcConville, MannevillePilaud-accordion}.
\begin{figure}[h]
	\capstart
	\centerline{\includegraphics[scale=1]{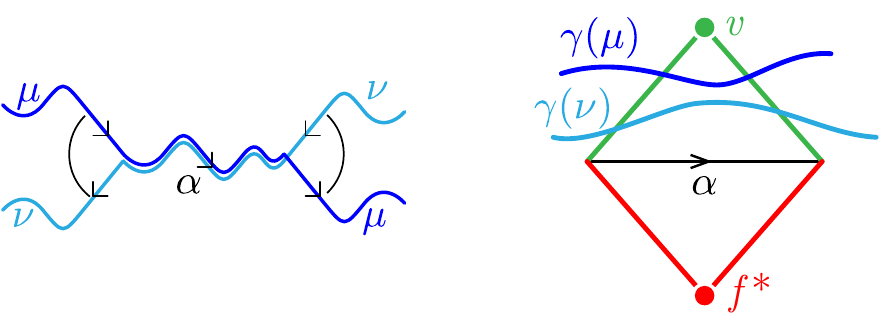}}
	\caption{The order~$\mu \prec_\alpha \nu$ on marked walks is the order in which the corresponding curves~$\curveof(\mu)$ and~$\curveof(\nu)$ cross the angles of~$\dissection$ and~$\dissection\dual$ corresponding to~$\alpha$.}
	\label{fig:orderCurves}
\end{figure}
\end{remark}

\begin{notation}
\label{notation:Falpha}
For any face~$F$ of~$\NKC$ and any arrow~$\alpha\in Q_1\blossom$, we denote by~$F_\alpha$ the set of walks of~$F$ marked at an occurence of~$\alpha^\pm$.
\end{notation}

\begin{lemma}
\label{lemm:countercurrentOrderIsTotal}
For any face~$F$ of~$\NKC$ and any arrow~$\alpha\in Q_1\blossom$, the countercurrent order~$\prec_\alpha$ defines a total order on~$F_{\alpha}$.
\end{lemma}

\begin{proof}
The proof of \cite[Lem.~2.21]{PaluPilaudPlamondon} applies here.
\end{proof}

\begin{lemma}
\label{lemm:facesHaveMaximalElement}
For any finite non-kissing face~$F$ of~$\NKC$ and any arrow~$\alpha\in Q_1\blossom$, the set~$F_\alpha$ is either empty or admits a maximal element for the countercurrent order~$\prec_\alpha$.
\end{lemma}

\begin{proof}
We provide two proofs of the lemma.  One is combinatorial and applies to~$\NKC$, and the other is geometric and applies to~$\NCC[\dissection_{\bar Q}, {\dissection\dual\!\!_{\bar Q}}]$, using the isomorphism of Theorem \ref{thm:complexesCoincide} and Remark \ref{rem:countercurrentOrderOnSurface}.

\para{Combinatorial proof}
Note that even though~$F$ is finite, the set of marked walks~$F_\alpha$ might be infinite if~$\alpha$ belongs to an oriented cycle of~$Q$.
Assume that~$F_\alpha$ is non-empty, and infinite.
Any walk that contains an infinite number of copies of~$\alpha$ can be written in the form~$\omega=\sigma c^{\pm\infty}$ where~$\sigma$ is a (possibly infinite) substring and~$c$ is an oriented cycle containing~$\alpha$ such that~$c, c^2\notin I$.
Then~$\omega$ marked at the first occurrence of~$\alpha$ in~$c^{\pm\infty}$ is greater than~$\omega$ marked at any other occurrence of~$\alpha$ in~$c^{\pm\infty}$.
We can thus safely remove all copies of~$\omega$ marked at an occurrence of~$\alpha$ in~$c^{\pm\infty}$ but the first one.
After applying this procedure a finite number of times,~$F_\alpha$ becomes finite, and the claim follows.

\para{Geometric proof}
View~$F$ as a face of the non-crossing complex~$\NCC[\dissection_{\bar Q}, {\dissection\dual\!\!_{\bar Q}}]$.
The arrow~$\alpha$ then corresponds to an angle in~$\dissection$ at a vertex~$v$.  
For any curve~$\gamma$ in~$F$, even though~$\gamma$ may pass through the above angle infinitely many times, there is one instance that is furthest from the angle: indeed, if ~$\gamma$ passes infinitely many times through~$\alpha$, then~$v$ is a puncture and~$\gamma$ circles infinitely towards it.
But then~$\gamma$ passes through~$\alpha$ only finitely many times before it starts circling~$v$ (in either direction), and one of these times is the furthest from~$v$.
Therefore, for every curve~$\gamma\in F$, there is a maximal occurence of~$\alpha$ in~$\gamma$.
Since~$F$ is finite, the set of these maximal occurences is finite, so it admits a maximal element.  
This element is the maximal element of~$F_\alpha$.
\end{proof}

\subsection{The non-crossing and non-kissing complexes have no infinite faces}

We now pause to prove the following statement.

\begin{proposition}
\label{prop:facetsAreFinite}
Any face of~$\NKC$ and~$\NCC[\dissection_{\bar Q}, {\dissection\dual\!\!_{\bar Q}}]$ is finite, thus contained in a finite maximal face.
\end{proposition}

Even though Theorem \ref{thm:complexesCoincide} ensures that~$\NKC$ and~$\NCC[\dissection_{\bar Q}, {\dissection\dual\!\!_{\bar Q}}]$ coincide, we provide two proofs.
One is combinatorial and applies to~$\NKC$, and the other is geometric and applies to~$\NCC[\dissection_{\bar Q}, {\dissection\dual\!\!_{\bar Q}}]$.
The proof requires the following intermediate results.

\begin{lemma}
\label{lem:nkCycles}
Let~$\omega$ be a non-kissing walk that contains a non-trivial oriented cycle~$c$ with~$c, c^2\notin I$.
Then~$\omega$ is of the form~$\sigma c^{\pm\infty}$, for some (possibly infinite) substring~$\sigma$.
\end{lemma}

\begin{proof}[Combinatorial proof]
Let~$c=\alpha_1\cdots\alpha_r$ be a non-trivial oriented cycle in some walk~$\omega$, with~$c, c^2\notin I$.
Assume that~$\omega$ is not of the form~$\sigma c^{\pm\infty}$.
Then, up to cyclic reordering of the arrows in~$c$, it contains a substring of the form~$\beta^\pm c^l \alpha_1\cdots\alpha_j\gamma^\pm$, for some arrows~$\beta,\gamma$ not in~$c$, some~$l\geq 1$, and some~$j<r$.
Since~$\bar Q$ is locally gentle, this substring is~$\beta^{-1} c^l \alpha_1\cdots\alpha_j\gamma^{-1}$.
Hence, the walk~$\omega$ contains the two substrings~$\beta^{-1}\alpha_1\cdots\alpha_j\alpha_{j+1}$ and~$\alpha_r\alpha_1\cdots\alpha_j\gamma^{-1}$, and thus self-kisses.

\para{Geometric proof}
An oriented cycle~$c$ corresponds to a puncture in~$\dissection_{\bar Q}$.
Let~$\gamma = \curveof(\omega)$.  The fact that~$\omega$ contains~$c$ translates to the fact that~$\gamma$ makes one full circle around the correponding puncture.

\begin{figure}[h]
	\capstart
	\centerline{\includegraphics[scale=.7]{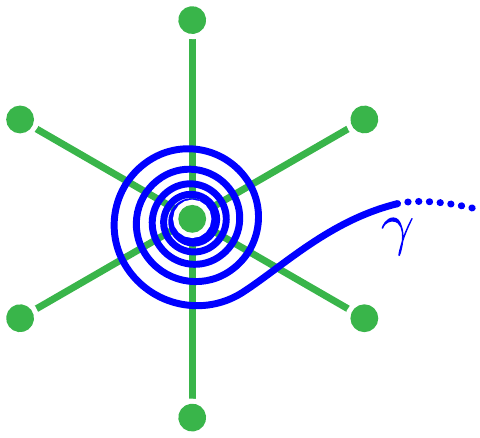}}
	\caption{The curve~$\gamma$ makes at least one full circle around the puncture.}
	\label{fig:aroundPuncture1}
\end{figure}

Once this cycle is completed, since~$\gamma$ cannot cross itself, it has to continue cycling infinitely around the puncture.  This proves the claim.
\end{proof}

\begin{lemma}\label{lem:ExistenceOfDistinguishedArrows}
Let~$F$ be a finite non-kissing face of~$\NKC$, let~$\omega$ be a walk in~$F$, and let~$\alpha$ be an arrow on~$\omega$.
Assume that~$\omega$ is not an infinite straight walk.
Then, there is an arrow~$\beta$ oriented in the same direction as~$\alpha$ in~$\omega$ such that~$\omega$ (marked at~$\beta$) is maximal in~$F$ for~$\prec_\beta$.
\end{lemma}

\begin{proof}
Let~$\omega$ be a walk in some non-kissing facet~$F$, marked at an arrow~$\alpha$.
If~$\omega$ is not maximal at~$\alpha$, there is a walk~$\mu_0\in F$ such that~$\omega\prec_\alpha \mu_0$.
Let~$\sigma_0$ be the maximal common substring of~$\omega$ and~$\mu_0$ containing~$\alpha$.
Since~$\omega$ and $\mu_0$ are different, they split at some endpoint of~$\sigma_0$.
At this endpoint, the arrow~$\beta$ of~$\omega$ not in~$\sigma_0$ points in the same direction as~$\alpha$ since~$\omega \prec_\alpha \mu_0$.
If~$\omega$ is maximal at~$\beta$ in~$F$, we are done.
Otherwise, let~$\mu_1$ be maximal in~$F$ for~$\prec_\beta$.
Since~$\mu_1$ does not kiss with~$\mu_0$, and~$\omega\prec_\beta\mu_1$,~$\mu_1$ contains~$\sigma_0$.
Let~$\sigma_1$ be the maximal common substring of~$\omega$ and~$\mu_1$ containing~$\beta$.
As remarked above,~$\sigma_1$ contains~$\sigma_0$.
Since~$\omega$ is not maximal at~$\beta$, the walks~$\omega$ and~$\mu_1$ split at some endpoint of~$\sigma_1$.
By Lemma~\ref{lem:nkCycles}, iterating this procedure eventually stops ($\mu_j=\omega$ for some $j$) unless~$\omega$ is an infinite straight walk.
\end{proof}

\begin{lemma}
\label{lemm:facesAreBounded}
There is a bound on the number of elements of a finite face of~$\NKC$.
\end{lemma}

\begin{proof}[Combinatorial proof]
By Lemma~\ref{lem:ExistenceOfDistinguishedArrows}, the number of walks in a finite face of~$\NKC$ is bounded by~$|Q_1\blossom|+p$, where~$p$ is the number of infinite straight walks (\ie the number of primitive oriented cycles in~$Q$).

\para{Geometric proof}
Let~$F$ be a face of the non-crossing complex~$\NCC[\dissection_{\bar Q}, {\dissection\dual\!\!_{\bar Q}}]$.
Its elements are pairwise compatible~$\dissection$-accordions on the surface~$\surface_{\bar Q}$.
Thus~$F$ can be viewed as a partial tagged triangulation of~$\surface_{\bar Q}$, by letting the~$\dissection$-accordions circling a puncture in (say) clockwise order be seen as ``notched'' arcs to said puncture, and the ones circling it in counterclockwise orientation be ``un-notched'' arcs.
It is known (see, for instance, \cite{FominShapiroThurston}) that any such partial tagged triangulation can be embedded into a full triangulation, and that triangulations contain a finite number of arcs.
Thus~$F$ is finite, and is contained in a finite maximal face.
\end{proof}

\begin{proof}[Proof of Proposition~\ref{prop:facetsAreFinite}]
If~$\NKC$, respectively~$\NCC$, admits an infinite face, then all subsets of this face are also faces, so~$\NKC$, respectively~$\NCC$, admits faces of arbitrary finite cardinality.
This contradicts Lemma \ref{lemm:facesAreBounded}.
\end{proof}

\subsection{Purity of the non-crossing and non-kissing complexes}

A simplicial complex is \defn{pure of dimension~$d$} if all its maximal faces have dimension~$d$.
The aim of this section is to prove the following purity result.

\begin{proposition}
\label{prop:purity}
The reduced non-kissing complex~$\RNKC$ and the reduced non-crossing complex~$\RNCC[\dissection_{\bar Q}, {\dissection\dual\!\!_{\bar Q}}]$ are pure of dimension~$|Q_0|$.
The non-kissing complex~$\NKC$ and the non-crossing complex~$\NCC[\dissection_{\bar Q}, {\dissection\dual\!\!_{\bar Q}}]$ are pure of dimension~$3|Q_0| - |Q_1| + p$ where~$p$ is the number of primitive oriented cycles in~$\bar Q$ or equivalently the number of punctures in~$V_{\bar Q}$.
\end{proposition}

Note that Proposition \ref{prop:purity} was proved for~$\NKC$ in \cite[Cor.~2.29]{PaluPilaudPlamondon} in the case where~$kQ/I$ is finite-dimensional.
Note also that, by Theorem \ref{thm:complexesCoincide} above,~$\NKC$ and~$\NCC[\dissection_{\bar Q}, {\dissection\dual\!\!_{\bar Q}}]$ coincide, so it suffices to prove the statement for one of them.
The following definition from \cite[Def.~2.25]{PaluPilaudPlamondon} extends to the setting of locally gentle algebras.

\begin{definition}
\label{def:distinguishedWalksAndArrows}
Let~$\alpha \in Q\blossom_1$ be an arrow and let~$F$ be a face of~$\NKC$ containing the straight walk passing through~$\alpha$.
Consider the marked walk~$\omega_\star \eqdef \max_{\prec_\alpha} F_\alpha$ (which exists by Lemma~\ref{lemm:facesHaveMaximalElement} and Proposition~\ref{prop:facetsAreFinite}),
and let~$\omega$ be the walk obtained by forgetting the marked arrow in~$\omega_\star$. 
We say that~$\omega$ is the \defn{distinguished walk} of~$F$ at~$\alpha$ and that~$\alpha$ is a \defn{distinguished arrow} on~$\omega \in F$.
We will use the notation
\[
\distinguishedWalk{\alpha}{F} \eqdef \max\nolimits_{\prec_\alpha} F_\alpha
\qquad\text{and}\qquad
\distinguishedArrows{\omega}{F} \eqdef \set{\alpha \in \omega}{\distinguishedWalk{\alpha}{F} = \omega}.
\]
\end{definition}

We use this definition to prove the following statement.

\begin{lemma}
\label{lem:walksThroughcycles}
For any oriented cycle~$c$ in~$Q$ with~$c, c^2\notin I$, any maximal non-kissing face contains a walk of the form~$\sigma c^{\varepsilon\infty}$, for some~$\varepsilon\in\{\pm1\}$ and some substring~$\sigma$ distinct from~$^{\varepsilon\infty}c$ (\ie the walk is not an infinite straight walk).
\end{lemma}

\begin{proof}[Combinatorial proof]
Consider a maximal non-kissing face~$F \in \NKC$ and an arrow~$\alpha$ in~$c$.
Let~$\beta, \gamma \in Q_1\blossom$ be such that~$t(\alpha) = s(\beta) = s(\gamma)$ and~$\alpha\beta \notin I$ while~$\alpha\gamma \in I$ (in other words, $\beta$ is the next arrow along~$c$ after~$\alpha$ and~$\gamma$ is the other outgoing arrow at~$t(\alpha)$).
Consider the walk~$\mu \eqdef \distinguishedWalk{\gamma}{F}$ and decompose it as~$\mu = \mu_1 \gamma \mu_2$ according to this occurence of~$\gamma$ on~$\mu$.
Consider the walk~$\nu \eqdef \mu_2^{-1} \gamma^{-1} c^\infty$.
If~$\nu$ is contained in~$F$, then we have found a walk of the desired form (remember that walks are considered as undirected).
Otherwise, there exists a walk~$\lambda$ of~$F$ preventing~$\nu$ from belonging to~$F$.
Let~$\sigma$ be a common substring of~$\nu$ and~$\lambda$ where they kiss.
Since~$\lambda$ must be compatible with~$\mu$ and with the straight walk~$c^\infty$, the substring~$\sigma$ must contain~${t(\alpha) = s(\beta)}$.
If~$\lambda$ contains neither~$\alpha$ nor~$\beta$, then it contains~$\gamma$ and enters $\nu$ at~$s(\gamma)$.
Since $\lambda$ kisses~$\nu$, it also enters~$\mu_2$, contradicting the maximality of~$\mu$ for~$\prec_\gamma$.
Therefore, $\lambda$ contains either~$\alpha$ or~$\beta$.
We have thus found a bending walk~$\lambda \in F$ containing an arrow~$\alpha$ or~$\beta$ of~$c$, say~$\alpha$.
We can now moreover assume that~$\lambda = \distinguishedWalk{\alpha}{F}$.
If~$\lambda$ starts by~$^\infty c$, or ends by~$c^\infty$, then we have found a walk of the desired form.
Otherwise, decompose this walk as~$\lambda = \lambda_1 \alpha \lambda_2$ and consider the walk~$\omega = \lambda_1 c^\infty$.
We claim that~$\omega$ belongs to~$F$.
Otherwise, there exists a walk~$\rho$ preventing~$\omega$ from belonging to~$F$.
Let~$\tau$ be a common substring of~$\omega$ and~$\rho$ where they kiss.
Since~$\omega$ must be compatible with~$\lambda$ and with the straight walk~$c^\infty$, the substring~$\tau$ must strictly contain the substring of~$\lambda$ along the cycle~$c$.
It follows that $\rho$ leaves~$\omega$ when it is on the cycle~$c$, thus with an incoming arrow, contradicting the maximality of~$\lambda$ for~$\prec_\alpha$.

\para{Geometric proof}
The statement of the lemma translates on~$\surface_{\bar Q}$ as follows: for any facet~$F$ of the non-crossing complex~$\NCC$ and any puncture of the dissection~$\dissection_{\bar Q}$, there exists a curve in~$F$ which goes to the puncture.
Assume that there is a puncture with no curves of~$F$ going to it.
Construct a curve~$\gamma$ as follows:~$\gamma$ first unwinds from the puncture, then~$\gamma$ leaves the local picture around the puncture by turning around another marked point (marked~$u$ on the picture below).

\begin{figure}[h]
	\capstart
	\centerline{\includegraphics[scale=.7]{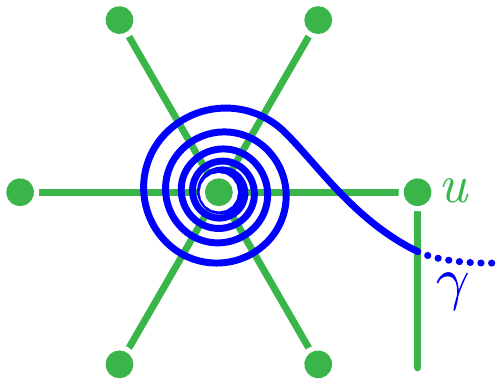}}
	\caption{The curve~$\gamma$ leaving the puncture.}
	\label{fig:aroundPuncture2}
\end{figure}

Two situations can occur:
\begin{itemize}
\item Firstly, if there are no curves of~$F$ closer to~$u$ than~$\gamma$, then~$\gamma$ can circle around~$u$ until it reaches the boundary (if~$u$ is on the boundary) or infinitely (if~$u$ is a puncture).
Then~$\gamma$ is a~$\dissection_{\bar Q}$-accordion which crosses none of those in~$F$, contradicting the maximality of~$F$.
\item Secondly, if there is a curve of~$F$ closer to~$u$ than~$\gamma$, then let~$\delta$ be the one furthest from~$u$ ($\delta$ exists by Lemma \ref{lemm:facesHaveMaximalElement}).
From there, let~$\gamma$ follow~$\delta$ on the surface until it either reaches the boundary or circles infinitely around a puncture.  
Then~$\gamma$ is compatible with all curves in~$F$, a contradiction.
\end{itemize}
Thus~$F$ contains a curve going to the puncture.
\end{proof}

The following result extends \cite[Prop.~2.28]{PaluPilaudPlamondon} to the case of locally gentle algebras.

\begin{proposition}
\label{prop:numberOfDistinguishedArrows}
In a non-kissing facet~$F$ of~$\NKC$,
\begin{itemize}
\item each bending walk of~$F$ has exactly two distinguished arrows,
\item each finite straight walk of~$F$ has exactly one distinguished arrow,
\item each infinite straight walk of~$F$ has no distinguished arrows.
\end{itemize}
\end{proposition}

\begin{proof}
The proof of \cite[Prop.~2.28]{PaluPilaudPlamondon} applies here to show that each walk in the non-kissing facet~$F$ has at most one distinguished arrow in each direction.
Lemma~\ref{lem:ExistenceOfDistinguishedArrows} therefore shows the first two points.
Finally, Lemma~\ref{lem:walksThroughcycles} shows the last point.
\end{proof}

We can now prove the purity result.

\begin{proof}[Proof of Proposition \ref{prop:purity}.]
The proof follows that of~\cite[Cor.~2.29]{PaluPilaudPlamondon}.
Let~$b$ denote the number of bending walks, $s$ the number of finite straight walks, and $p$ the number of infinite straight walks in a non-kissing facet~$F$ of~$\NKC$.
Since each finite straight walk uses two blossom arrows, we have~$s = 2|Q_0|-|Q_1|$ (see Remark~\ref{rem:sizeBlossomingQuiver}).
Moreover, the infinite straight walks are just the primitive oriented cycles in~$\bar Q$.

Using Proposition~\ref{prop:numberOfDistinguishedArrows}, we obtain by double-counting that
\[
2 * b + 1 * s + 0 * p= |Q_1\blossom| = 4|Q_0|-|Q_1|
\]
from which we derive that~$b = |Q_0|$ and~$b+s+p = 3|Q_0|-|Q_1|+p$.
\end{proof}

\subsection{Thinness of the non-crossing and non-kissing complexes, mutation}

A simplicial complex is \defn{thin} if each of its codimension~$1$ faces is contained in exaclty two facets.
We state the following results, which generalize \cite[Prop.~2.33 \& Cor.~2.35]{PaluPilaudPlamondon}, whose proofs apply without change.

\begin{proposition}
\label{prop:flip}
Let~$F$ be a non-kissing facet in~$\NKC$ and let~$\omega \in F$ be a bending walk.
Let~$\alpha$ and~$\beta$ be the distinguished arrows of~$\omega$ (see Proposition \ref{prop:numberOfDistinguishedArrows}), and~$\sigma$ be the distinguished substring of~$\omega$, which splits~$\omega$ into~$\omega = \rho \sigma \tau$.
Let~$\alpha'$ and~$\beta'$ be the other two arrows of~$Q_1\blossom$ incident to the endpoints of~$\sigma$ and such that~$\alpha'\alpha \in I$ or~$\alpha\alpha' \in I$, and $\beta'\beta \in I$ or~$\beta\beta' \in I$.
Let~$\mu \eqdef \distinguishedWalk{\alpha'}{F \ssm \{\omega\}}$ and~$\nu \eqdef \distinguishedWalk{\beta'}{F \ssm \{\omega\}}$ be the distinguished walks of~$F \ssm \{\omega\}$ at~$\alpha'$ and~$\beta'$ respectively.
Then
\begin{enumerate}[(i)]
\item The walk~$\mu$ splits into~$\mu = \rho' \sigma \tau$ and the walk~$\nu$ splits into~$\nu = \rho \sigma \tau'$.
\item The walk~$\omega' \eqdef \rho' \sigma \tau'$ is kissing~$\omega$ but no other~walk~of~$F$. Moreover, $\omega'$ is the only other (undirected) walk besides~$\omega$ which is not kissing any other walk of~$F \ssm \{\omega\}$.
\end{enumerate}
\end{proposition}
The walk~$\omega'$ of Proposition \ref{prop:flip} is the \defn{flip} or \defn{mutation} of~$F$ at~$\omega$.
The flip is \defn{increasing} if~$\sigma$ is on top of~$\omega$, and \defn{decreasing} otherwise.

\begin{figure}[t]
	\capstart
	\centerline{\includegraphics[scale=1]{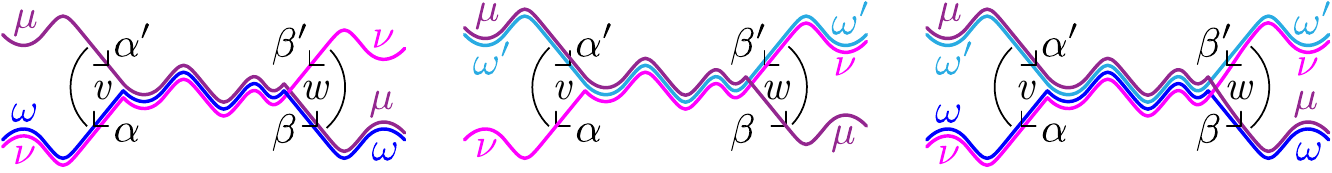}}
	\caption{A schematic representation of flips in the non-kissing complex. The flip exchanges the walk~$\omega = \rho \sigma \tau$ in the facet~$F$ (left) with the walk~$\omega' = \rho' \sigma \tau'$ in the facet~$F'$ (middle), using the walks~$\mu = \rho' \sigma \tau$ and~$\nu = \rho \sigma \tau'$. The walk~$\omega'$ kisses~$\omega$ but no other walk of~$F$ (right).}
	\label{fig:flip1}
\end{figure}

\begin{corollary}
\label{coro:thin}
The non-kissing complex~$\RNKC$ is thin and is thus a pseudomanifold without boundary.
\end{corollary}

The mutation of non-kissing facets translates on~$\surface(\bar Q)$ as follows.

\begin{proposition}
\label{prop:mutationAccordions}
Let~$F$ be a facet of the non-crossing complex~$\NCC$, and let~$\gamma \in F$.  We represent all curves on the universal cover of~$\surface(\bar Q)$.
Then~$F$ forms a~$B$-dissection of the surface, and~$\gamma$ lies in at least one of two configutations~$\ZZZ$ and~$\SSS$ with the other curves of the dissection (where the endpoints may go to infinity if punctures are involved).  See \fref{fig:flip3}.
\begin{itemize}
\item Exactly one of the curves obtained by changing the~$\SSS$ to a~$\ZZZ$ (and conversely) by moving the endpoints of~$\gamma$ along the top and bottom curves is a~$\dissection$-accordion.
\item This~$\dissection$-accordion yields the flip or mutation of~$F$ at~$\gamma$.
\item If the flip is increasing, then the~$\ZZZ$ is turned into a~$\SSS$.  If the flip is decreasing, then the~$\SSS$ is turned in to a~$\ZZZ$.
\end{itemize}
\begin{figure}[h]
	\capstart
	\centerline{\includegraphics[scale=1]{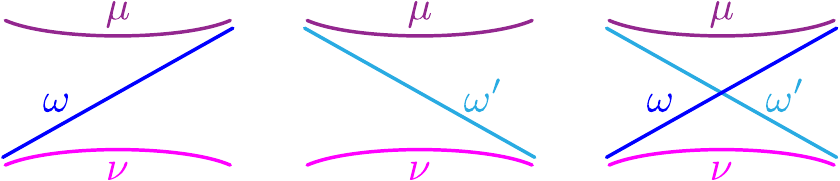}}
	\caption{An increasing flip turns a~$\ZZZ$ (formed by~$\mu$, $\omega$ and~$\nu$) to a~$\SSS$ (formed by~$\mu$, $\omega'$ and~$\nu$). The two curves~$\omega$ and~$\omega'$ are crossing.}
	\label{fig:flip3}
\end{figure}
\end{proposition}

\begin{proof}
Let~$\omega = \walk(\gamma)$.  Using the notation of Proposition \ref{prop:flip}, the local configuration around~$\omega$ (as in Figures~\ref{fig:flip1} and~\ref{fig:flip2}) translates on the (universal cover of the) surface~$\surface(\bar Q)$ as follows.
\begin{figure}[h]
	\capstart
	\centerline{\includegraphics[scale=1]{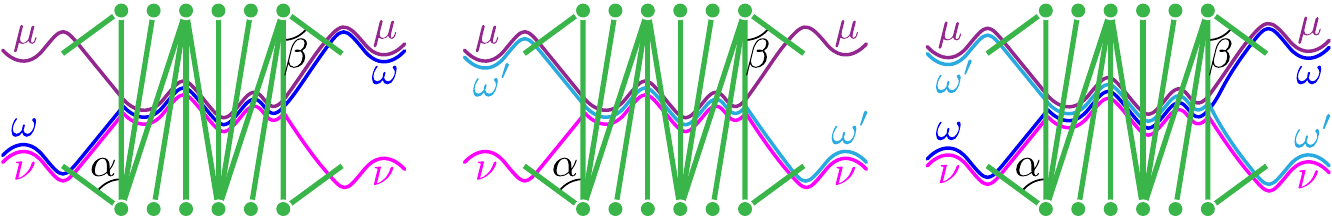}}
	\caption{A schematic representation of flips in the non-crossing complex. The flip exchanges the curve~$\omega$ in the facet~$F$ (left) with the curve~$\omega'$ in the facet~$F'$ (middle), using the curves~$\mu$ and~$\nu$. The curve~$\omega'$ crosses~$\omega$ but no other curve of~$F$ (right).}
	\label{fig:flip2}
\end{figure}
We see that the only two~$\dissection$-accordions found between~$\mu$ and~$\curveof(\omega)$ are~$\mu$ and~$\curveof(\omega)$ themselves.
The same holds by replacing~$\mu$ with~$\nu$.
Thus~$\mu, \nu$ and~$\curveof(\omega)$ lie in a~$\SSS$ or a~$\ZZZ$.
We see that~$\curveof(\omega')$ is then obtained as claimed. The direction of the flip translates from the picture in a similar way.
\end{proof}

\section{$\b{g}$-vectors, $\b{c}$-vectors, and $\b{d}$-vectors}
\label{sec:vectors}

In this section, we define three relevant families of vectors associated to walks in the non-kissing complex of a locally gentle quiver~$\bar Q$.
In contrast to the $\b{d}$-vectors, the $\b{g}$- and $\b{c}$-vectors were already considered in the non-kissing complex of a gentle quiver~\cite[Sect.~4.2]{PaluPilaudPlamondon}, and their definition immediately generalizes to the case of locally gentle algebras.
All these vectors have algebraic motivations that we prefer to omit in the present paper to remain focused on the two-sided combinatorics of the non-kissing and non-crossing models.
Our objective is to provide interpretations of these vectors in terms of the surface~$\surface_{\bar Q}$.

\subsection{$\b{g}$-vectors, $\b{c}$-vectors, and $\b{d}$-vectors for walks}
\label{subsec:gcdVectorsWalks}

Let~$\bar Q = (Q,I)$ be a locally gentle bound quiver, and let~$(\b{e}_a)_{a \in Q_0}$ denote the standard basis of~$\R^{Q_0}$.

\begin{definition}
For a multiset~$V \eqdef \{a_1, \dots, a_m\}$ of vertices of~$Q_0$, we denote by~$\multiplicityVector_V \in \R^{Q_0}$ the \defn{multiplicity vector} of~$V$, defined by
\[
\multiplicityVector_V \eqdef \sum_{i \in [m]} \b{e}_{a_i} = \sum_{a \in Q_0} |\set{i \in [m]}{a_i = a}| \, \b{e}_a.
\]
For a string~$\sigma$ of~$\bar Q$, we define~$\multiplicityVector_\sigma \eqdef \multiplicityVector_{V(\sigma)}$ where~$V(\sigma)$ is the multiset of vertices of~$\sigma$. 
\end{definition}

We now recall two definitions from~\cite{PaluPilaudPlamondon} that still make sense for locally gentle algebras.

\begin{definition}
\label{def:gVectorsWalks}
For any walk~$\omega$ of~$\bar Q$, we let~$\peaks{\omega}$ (resp.~by~$\deeps{\omega}$) be the multiset of vertices of~$Q_0$ corresponding to the peaks (resp.~deeps) of~$\omega$.
The \defn{$\b{g}$-vector} of the walk~$\omega$ is the vector~${\gvector{\omega} \in \R^{Q_0}}$ defined by
\[
\gvector{\omega} \eqdef \multiplicityVector_{\peaks{\omega}} - \multiplicityVector_{\deeps{\omega}}.
\]
For a set~$\Omega$ of walks, we let~$\gvectors{\Omega} \eqdef \set{\gvector{\omega}}{\omega \in \Omega}$.
Observe that, if~$\omega$ is a straight walk, then~$\gvector{\omega} = 0$. Moreover, since any infinite walk~$\omega$ is eventually cyclic, the coordinates of~$\gvector{\omega}$ are always finite.
\end{definition}

\begin{definition}
\label{def:cVectorsWalks}
Let~$F \in \RNKC$ be a non-kissing facet and $\omega$ be a walk in~$F$.
Define
\[
\distinguishedSign{\omega}{F} \eqdef 
\begin{cases}
\phantom{-}1 & \text{if~$\distinguishedString{\omega}{F}$ is a top substring of~$\omega$ (\ie if~$\distinguishedArrows{\omega}{F}$ point outside),} \\
-1 & \text{if~$\distinguishedString{\omega}{F}$ is a bottom substring of~$\omega$ (\ie if~$\distinguishedArrows{\omega}{F}$ point inside),}
\end{cases}
\]
where~$\distinguishedString{\omega}{F}$ and~$\distinguishedArrows{\omega}{F}$ are the distinguished arrows of~$\omega$ in~$F$ (see Definition~\ref{def:distinguishedWalksAndArrows}).
The \defn{$\b{c}$-vector} of a walk~$\omega$ of~$F$ is the vector~${\cvector{\omega}{F} \in \R^{Q_0}}$ defined by
\[
\cvector{\omega}{F} \eqdef \distinguishedSign{\omega}{F} \, \multiplicityVector_{\distinguishedString{\omega}{F}}.
\]
We let~$\cvectors{F} \eqdef \set{\cvector{\omega}{F}}{\omega \in F}$ be the set of $\b{c}$-vectors of a non-kissing facet~$F \in \RNKC$, and~$\allcvectors \eqdef \bigcup_F \cvectors{F}$ be the set of all $\b{c}$-vectors of all non-kissing facets~${F \in \RNKC}$.
\end{definition}

We now introduce another algebraically motivated family of vectors that was omitted in~\cite{PaluPilaudPlamondon}.
We first need to introduce the kissing number.

\begin{definition}
\label{def:kissingNumber}
For two walks~$\omega$ and~$\omega'$ of~$\bar Q$, denote by~$\kn(\omega,\omega')$ the number of times that~$\omega$ kisses $\omega'$.
Note that kisses are counted with multiplicities if~$\omega$ or~$\omega'$ pass twice through the same substring.
Observe also that~$\kn(\omega,\omega')$ is finite since any infinite walk is eventually cyclic.
The \defn{kissing number} of~$\omega$ and~$\omega'$ is~$\KN(\omega,\omega') \eqdef \kn(\omega,\omega') + \kn(\omega',\omega)$.
\end{definition}

\begin{definition}
\label{def:dVectorsWalks}
The \defn{$\b{d}$-vector} of a walk~$\omega$ is the vector~$\dvector{\omega} \in \R^{Q_0}$ defined by
\[
\dvector{\omega} \eqdef 
\begin{cases}
- \b{e}_a & \text{if $\omega$ is the deep walk~$a_\deep$ with a unique deep at~$a$,} \\
\displaystyle \sum_{a \in Q_0} \kn(\omega, a_\deep) \, \b{e}_a & \text{otherwise.}
\end{cases}
\]
For a set~$\Omega$ of walks, we let~$\dvectors{\Omega} \eqdef \set{\dvector{\omega}}{\omega \in \Omega}$.
\end{definition}

\begin{example}
\label{exm:gcdVectorsWalks}
For instance, the $\b{g}$-, $\b{c}$-, and $\b{d}$-matrices of the facet illustrated in \fref{fig:gcdVectors} are given by
\[
\gvectors{F} = 
\begin{blockarray}{ccc}
	{\color{blue} \bullet} & {\color{violet} \bullet} & \\[-.1cm]
	\begin{block}{[cc]c}
	1 & 0 & \!\!\!\circled{1} \\
	-1\!\! & \!\!-1 & \!\!\!\circled{2} \\
	\end{block}
\end{blockarray}
\qquad\qquad
\cvectors{F} = 
\begin{blockarray}{ccc}
	{\color{blue} \bullet} & {\color{violet} \bullet} & \\[-.1cm]
	\begin{block}{[cc]c}
	1 & \!\!-1 & \!\!\!\circled{1} \\
	0 & \!\!-1 & \!\!\!\circled{2} \\
	\end{block}
\end{blockarray}
\qquad\qquad
\dvectors{F} = 
\begin{blockarray}{ccc}
	{\color{blue} \bullet} & {\color{violet} \bullet} & \\[-.1cm]
	\begin{block}{[cc]c}
	1 & 0 & \!\!\!\circled{1} \\
	0 & -1 & \!\!\!\circled{2} \\
	\end{block}
\end{blockarray}
\]
(\ie columns are the $\b{g}$-, $\b{c}$-, $\b{d}$-vectors of the elements of the facet).

\begin{figure}[t]
	\capstart
	\centerline{\includegraphics[scale=.7]{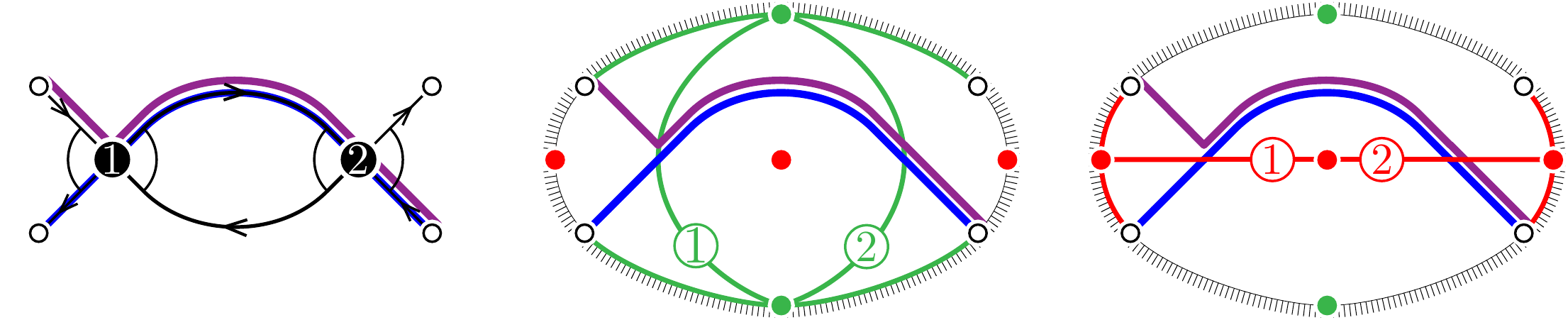}}
	\caption{A non-kissing facet of~$\RNKC$ (left), an accordion facet~$\RAC$ (middle) and a slalom facet~$\RSC$ (right) that correspond via the bijections of Theorem~\ref{thm:complexesCoincide} and Proposition~\ref{prop:accordionsSlaloms}. The~$\b{g}$-, $\b{c}$-, and $\b{d}$-matrices of this facet are given in Example~\ref{exm:gcdVectorsWalks}.}
	\label{fig:gcdVectors}
\end{figure}
\end{example}

\subsection{$\b{g}$-vectors, $\b{c}$-vectors, and~$\b{d}$-vectors on the surface}
\label{subsec:gcdVectorsSurface}

Our objective is now to interpret the three families of vectors of Section~\ref{subsec:gcdVectorsWalks} in terms of accordions and slaloms on the surface.
According to Remark~\ref{rem:propertiesSurface}\,(\eqref{item:edges}), each vertex~$a$ of~$Q$ corresponds to an edge~$\edgeof(a)$ in the dissection~$\dissection_{\bar Q}$ and  to an edge~$\dualedgeof(a) = \edgeof(a)\dual$ in the dual dissection~${\dissection\dual\!\!_{\bar Q}}$.
Recall moreover from Lemma~\ref{lemm:curveOfAWalkIsAccordion} and Proposition~\ref{prop:accordionsSlaloms} that each non-kissing walk~$\omega$ in~$\bar Q$ corresponds to a curve~$\curveof(\omega)$ that is both a $\dissection$-accordion and a $\dissection\dual$-slalom.
The following three statements translate the definitions of $\b{g}$-, $\b{c}$- and $\b{d}$-vectors on the surface.
The proofs are left to the reader.
We note that these interpretations were already used in~\cite{MannevillePilaud-accordion, GarverMcConville} when the surface is a disk.
When~$\dissection$ is a triangulation of~$\surface$, these interpretations coincide with those given in terms of triangulations and laminations for cluster algebras from surfaces by S.~Fomin and D.~Thurston~\cite{FominThurston}.

\begin{proposition}
\label{prop:gVectorsSurface}
For a walk~$\omega$ of~$\bar Q$ and a vertex~$a \in Q_0$, the $a$-th entry of the $\b{g}$-vector~$\gvector{\omega}$ is given~by:
\begin{enumerate}[(i)]
\item The number of times that the $\dissection$-accordion~$\curveof(\omega)$ positively crosses the edge~$\edgeof(a)$ of~$\dissection$ minus the number of times it negatively crosses the edge~$\edgeof(a)$ of~$\dissection$. Here, the crossing is positive (resp.~negative) if the two angles of~$\dissection$ crossed by~$\curveof(\alpha)$ just before and after~$\edgeof(a)$ form a~$\SSS$ (resp.~a~$\ZZZ$).
We note that some crossings are neither positive nor negative.
\item The number of times that the $\dissection\dual$-slalom~$\curveof(\omega)$ positively crosses the edge~$\dualedgeof(a)$ of~$\dissection\dual$ minus the number of times it negatively crosses the edge~$\dualedgeof(a)$ of~$\dissection\dual$. Here, the crossing is positive when~$\curveof(\omega)$ leaves a face~$u\dual$ to enter a face~$v\dual$ with the vertex~$u$ being to its right (and thus with~$v$ being to its left), and the crossing is negative otherwise.
We note that, in this case, each crossing is either positive or negative.
\end{enumerate}
\end{proposition}

\begin{proposition}
\label{prop:cVectorsSurface}
Consider a walk~$\omega$ of~$\bar Q$ in a non-kissing facet~$F \in \RNKC$.
Call \defn{distinguished subcurve} of~$\curveof(\omega)$ in~$\curveof(F)$ the part of~$\curveof(\omega)$ which lies strictly between the two angles of~$\dissection$ (or of~$\dissection\dual$) where~$\curveof(\omega)$ is maximal (see Remark~\ref{rem:countercurrentOrderOnSurface}).
For a vertex~$a \in Q_0$, the $a$-th entry of the $\b{c}$-vector~$\cvector{\omega}{F}$ is up to the sign given by:
\begin{enumerate}[(i)]
\item the number of crossings between the distinguished subcurve of~$\curveof(\omega)$ in~$\curveof(F)$ and the edge~$\edgeof(a)$ of the dissection~$\dissection$,
\item the number of times the distinguished subcurve of~$\curveof(\omega)$ in~$\curveof(F)$ follows the edge~$\dualedgeof(a)$ of the dual dissection~$\dissection\dual$. Here, we say that a curve~$\delta$ follows~$\dualedgeof(a)$ if~$\dualedgeof(a)$ lies on the boundary opposite to~$u$ of a face~$u\dual$ of~$\dissection\dual$ crossed by~$\delta$.
\end{enumerate}
\end{proposition}

\begin{remark}
Observe that although we gave interpretations for the $\b{g}$- and~$\b{c}$-vectors both in terms of accordions and slaloms, we consider that the slaloms naturally correspond to~$\b{g}$-vectors, while the accordions are more suited for~$\b{c}$-vectors.
\end{remark}

\begin{proposition}
\label{prop:dVectorsSurface}
Denote by~$X_\deep \eqdef \set{\curveof(a_\peak)}{a \in Q_0}$ the deep facet of the non-crossing complex obtained by slightly rotating~$\dissection_{\bar Q}$ counterclockwise as described in Example~\ref{exm:peakDeepFacets}.
For a walk~$\omega$ of~$\bar Q$ and a vertex~$a \in Q_0$, the $a$-th entry of the $\b{d}$-vector~$\dvector{\omega}$ is either~$-1$ if~$\omega = a_\peak$ or the number of crossings of the curve~$\gamma(\omega)$ with the edge~$\curveof(a_\peak)$ of~$X_\deep$ corresponding to~$a$.
\end{proposition}

\subsection{Geometric properties of these vectors}
\label{subsec:geometricProperties}

We now gather some statements about $\b{g}$-, $\b{c}$-, and~$\b{d}$-vectors, essentially borrowed from~\cite[Part.~4]{PaluPilaudPlamondon}.

\begin{proposition}
\label{rem:signCoherence}
The $\b{g}$-, $\b{c}$-, and~$\b{d}$-vectors have the \defn{sign-coherence property}: for any non-kissing facet~$F \in \RNKC$,
\begin{itemize}
\item for~$a \in Q_0$, the~$a$-th coordinates of all $\b{g}$-vectors~$\gvector{\omega}$ for~$\omega \in F$ have the same sign,
\item for any~$\omega \in F$, all coordinates of the $\b{c}$-vector~$\cvector{\omega}{F}$ have the same sign,
\item for any~$\omega \in F$, all coordinates of the $\b{d}$-vector~$\dvector{\omega}$ have the same sign.
\end{itemize}
\end{proposition}

\begin{proof}
For the $\b{g}$-vectors, observe that two walks~$\omega$ and~$\omega'$ whose $\b{g}$-vectors~$\gvector{\omega}$ and~$\gvector{\omega'}$ have opposite signs on their~$a$-th coordinate automatically kiss at~$a$.
For $\b{c}$- and $\b{d}$-vectors, the statement is immediate from the definition.
\end{proof}

\begin{proposition}
\label{prop:gvectorscvectorsDualBases}
For any non-kissing facet~$F \in \RNKC$, the set of $\b{g}$-vectors~$\gvectors{F}$ and the set of $\b{c}$-vectors~$\cvectors{F}$ form dual bases. 
\end{proposition}

\begin{proof}
Noting that the distinguished string of a walk in a given non-kissing facet is always a finite substring, the proof of~\cite[Proposition 4.16]{PaluPilaudPlamondon} applies \emph{verbatim}.
\end{proof}

\begin{example}
For instance, one can check that~$\gvectors{F} \cdot \transpose{\cvectors{F}} = 1$ in Example~\ref{exm:gcdVectorsWalks}.
\end{example}

Finally, we borrow the following geometric statement from~\cite{PaluPilaudPlamondon}.
We refer to \cite{Ziegler-polytopes} for basic notions on polyhedral geometry (polytopes and normal fans).

\begin{theorem}
\label{thm:fanAssociahedron}
Let~$\bar Q$ be a locally gentle bound quiver such that~$\NKC$ is finite. The collection of cones
\[
\gvectorFan \eqdef \bigset{\R_{\ge0} \gvectors{F}}{F \text{ non-kissing face of } \RNKC}.
\]
is a complete simplicial fan, called the \defn{$\b{g}$-vector fan} of~$\bar Q$.
Moreover, $\gvectorFan$ is the normal fan of the \defn{$\bar Q$-associahedron} defined equivalently as
\begin{itemize}
\item the convex hull of the points~$\point{F} \eqdef \sum_{\omega \in F} \KN(\omega) \, \cvector{\omega}{F}$ for all facets~${F \in \RNKC}$,~or
\item the intersection of the halfspaces~$\HS{\omega} \eqdef \bigset{\b{x} \in \R^{Q_0}}{\dotprod{\gvector{\omega}}{\b{x}} \le \KN(\omega)}$ for all walks~$\omega$ on~$\bar Q$,
\end{itemize}
where $\KN(\omega) \eqdef \sum_{\omega'} \KN(\omega,\omega')$ is the (finite) sum of the kissing numbers of Definition~\ref{def:kissingNumber}.
\end{theorem}

\begin{proof}
The proof is identical to that of Theorem 4.17 and Theorem 4.27 of~\cite{PaluPilaudPlamondon}.
\end{proof}

\fref{fig:fanAssociahedron} illustrates Theorem~\ref{thm:fanAssociahedron}.

\begin{figure}[p]
	\capstart
	\centerline{\includegraphics[scale=.5]{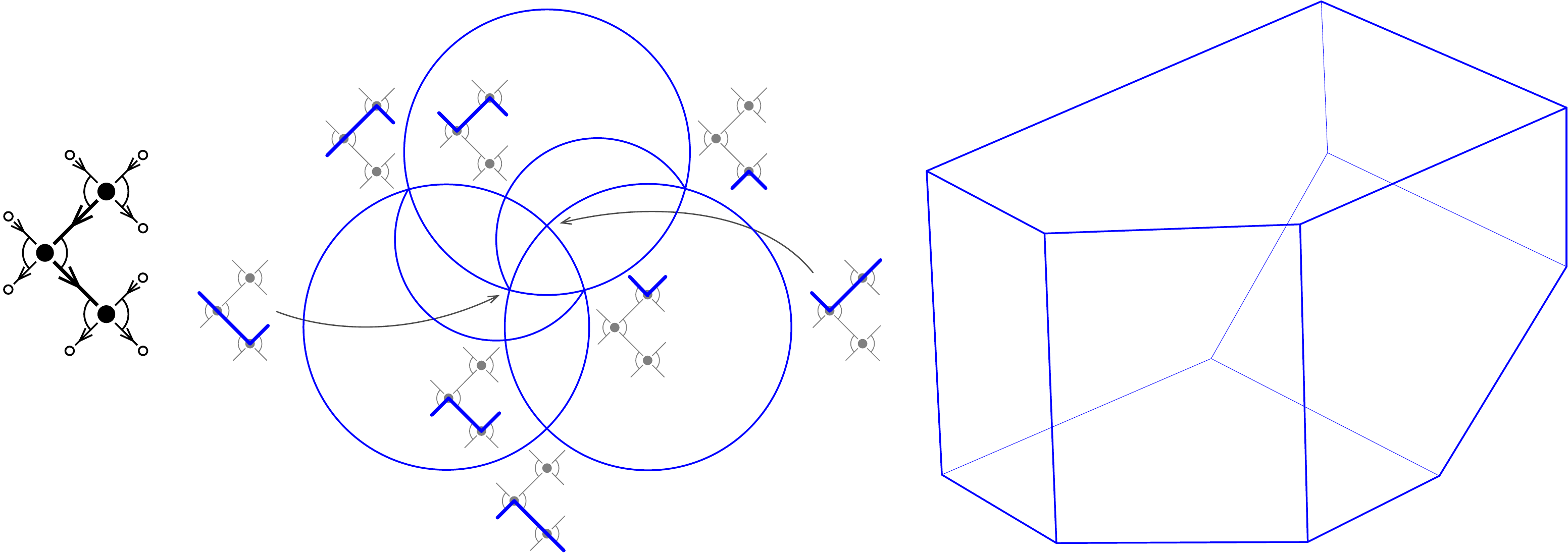}}
	\vspace{.5cm}
	\centerline{\includegraphics[scale=.5]{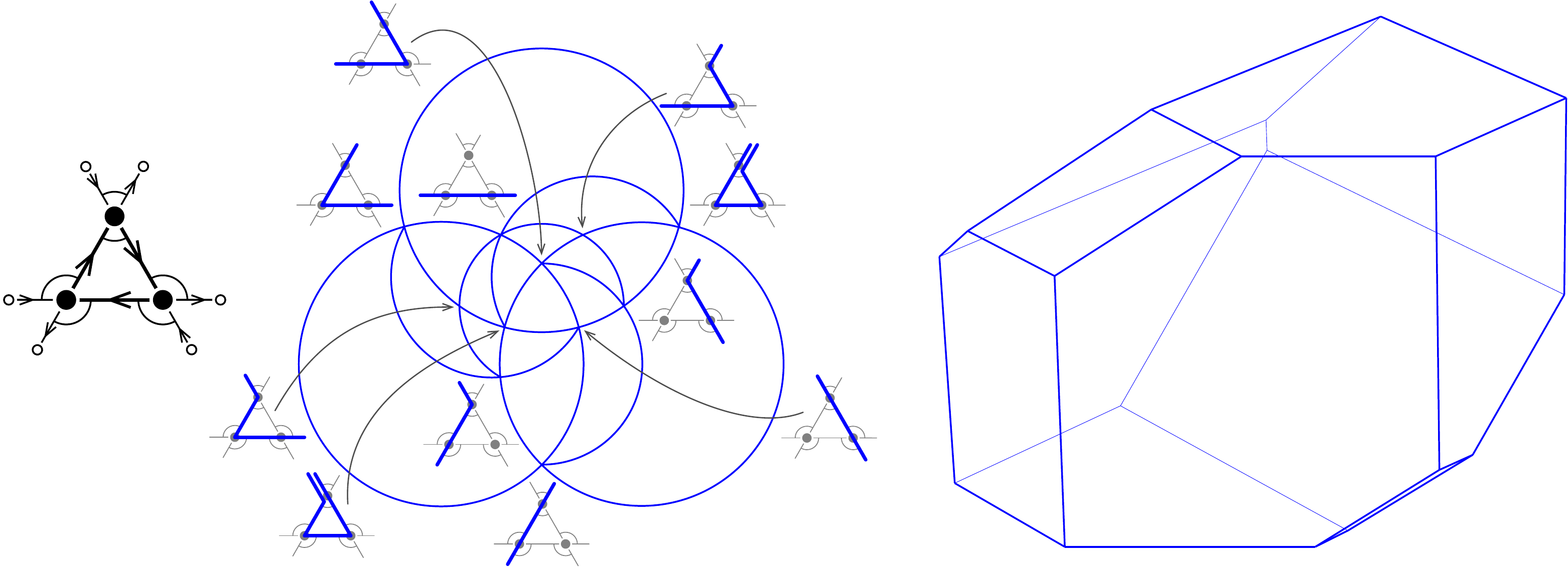}}
	\vspace{.5cm}
	\centerline{\includegraphics[scale=.5]{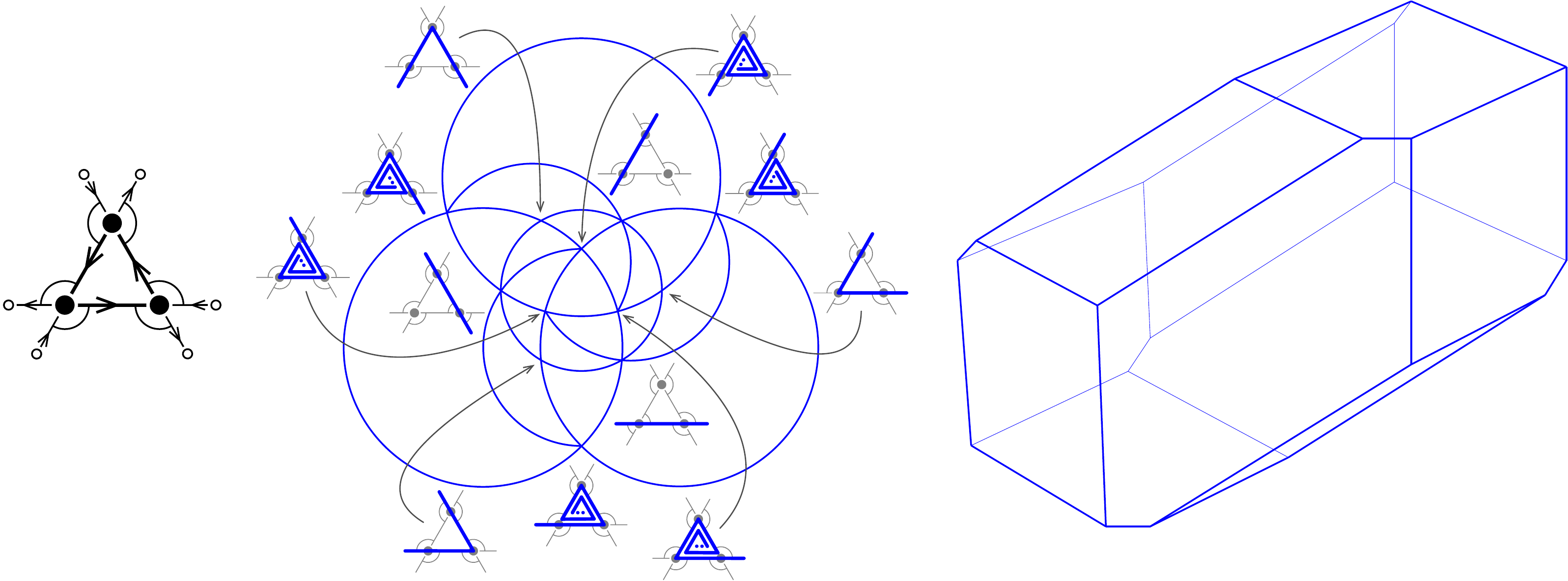}}
	\caption{The $\b{g}$-vector fan~$\gvectorFan$ (left) and the $\bar Q$-associahedron~$\Asso$ (right) of three locally gentle quivers. The first two are gentle and the illustration is borrowed from~\cite{PaluPilaudPlamondon}. The last is locally gentle but not gentle.}
	\label{fig:fanAssociahedron}
\end{figure}


\bibliographystyle{alpha}
\bibliography{accordionComplexSurfaces}

\newcommand{\etalchar}[1]{$^{#1}$}
\begin{thebibliography}{ABCJP10}

\bibitem[ABCJP10]{ABCP}
Ibrahim Assem, Thomas Br\"ustle, Gabrielle Charbonneau-Jodoin, and Pierre-Guy
  Plamondon.
\newblock Gentle algebras arising from surface triangulations.
\newblock {\em Algebra Number Theory}, 4(2):201--229, 2010.

\bibitem[AIR14]{AdachiIyamaReiten}
Takahide Adachi, Osamu Iyama, and Idun Reiten.
\newblock {$\tau$}-tilting theory.
\newblock {\em Compos. Math.}, 150(3):415--452, 2014.

\bibitem[Bar01]{Baryshnikov}
Yuliy Baryshnikov.
\newblock On {S}tokes sets.
\newblock In {\em New developments in singularity theory ({C}ambridge, 2000)},
  volume~21 of {\em NATO Sci. Ser. II Math. Phys. Chem.}, pages 65--86. Kluwer
  Acad. Publ., Dordrecht, 2001.

\bibitem[BCS18]{BaurCoelhoSimoes}
Karin Baur and Raquel Coelho Sim\~oes.
\newblock A geometric model for the module category of a gentle algebra.
\newblock Preprint,
  \href{http://arxiv.org/abs/1803.05802}{\texttt{arXiv:1803.05802}}, 2018.

\bibitem[BDM{\etalchar{+}}17]{BrustleDouvilleMousavandThomasYildirim}
Thomas Br\"ustle, Guillaume Douville, Kaveh Mousavand, Hugh Thomas, and Emine
  Y\i{}ld\i{}r\i{}m.
\newblock On the combinatorics of gentle algebras.
\newblock Preprint,
  \href{http://arxiv.org/abs/1707.07665}{\texttt{arXiv:1707.07665}}, 2017.

\bibitem[BH08]{BessenrodtHolm}
Christine Bessenrodt and Thorsten Holm.
\newblock Weighted locally gentle quivers and {C}artan matrices.
\newblock {\em J. Pure and Appl. Algebra}, 212:204--221, 2008.

\bibitem[BR87]{ButlerRingel}
M.~C.~R. Butler and Claus~Michael Ringel.
\newblock Auslander-{R}eiten sequences with few middle terms and applications
  to string algebras.
\newblock {\em Comm. Algebra}, 15(1-2):145--179, 1987.

\bibitem[CB18]{Crawley-Boevey}
William Crawley-Boevey.
\newblock Classification of modules for infinite-dimensional string algebras.
\newblock {\em Trans. Amer. Math. Soc.}, 370(5):3289--3313, 2018.

\bibitem[CCS06]{CalderoChapotonSchiffler}
P.~Caldero, F.~Chapoton, and R.~Schiffler.
\newblock Quivers with relations arising from clusters ({$A_n$} case).
\newblock {\em Trans. Amer. Math. Soc.}, 358(3):1347--1364, 2006.

\bibitem[Cha16]{Chapoton-quadrangulations}
Fr\'ed\'eric Chapoton.
\newblock Stokes posets and serpent nests.
\newblock {\em Discrete Math. Theor. Comput. Sci.}, 18(3), 2016.

\bibitem[DRS12]{DavidRoeslerSchiffler}
Lucas David-Roesler and Ralf Schiffler.
\newblock Algebras from surfaces without punctures.
\newblock {\em J. Algebra}, 350:218--244, 2012.

\bibitem[FHS82]{FreedmanHassScott}
Michael Freedman, Joel Hass, and Peter Scott.
\newblock Closed geodesics on surfaces.
\newblock {\em Bulletin of the London Mathematical Society}, 14(5):385--391,
  1982.

\bibitem[FST08]{FominShapiroThurston}
Sergey Fomin, Michael Shapiro, and Dylan Thurston.
\newblock Cluster algebras and triangulated surfaces. {P}art {I}: {C}luster
  complexes.
\newblock {\em Acta Math.}, 201(1):83--146, 2008.

\bibitem[FT18]{FominThurston}
Sergey Fomin and Dylan Thurston.
\newblock Cluster algebras and triangulated surfaces. part {II}: {L}ambda
  lengths.
\newblock {\em Memoirs AMS}, 255(1223), 2018.

\bibitem[GM18]{GarverMcConville}
Alexander Garver and Thomas McConville.
\newblock Oriented flip graphs of polygonal subdivisions and noncrossing tree
  partitions.
\newblock {\em J. Combin. Theory Ser. A}, 158:126--175, 2018.

\bibitem[HKK17]{HaidenKatzarkovKontsevich}
F.~Haiden, L.~Katzarkov, and M.~Kontsevich.
\newblock Flat surfaces and stability structures.
\newblock {\em Publ. Math. Inst. Hautes \'Etudes Sci.}, 126:247--318, 2017.

\bibitem[LF09]{Labardini}
Daniel Labardini-Fragoso.
\newblock Quivers with potentials associated to triangulated surfaces.
\newblock {\em Proc. Lond. Math. Soc. (3)}, 98(3):797--839, 2009.

\bibitem[LP18]{LekiliPolishchuk}
Yanki Lekili and Alexander Polishchuk.
\newblock Derived equivalences of gentle algebras via {F}ukaya categories.
\newblock Preprint,
  \href{http://arxiv.org/abs/1801.06370}{\texttt{arXiv:1801.06370}}, 2018.

\bibitem[McC17]{McConville}
Thomas McConville.
\newblock Lattice structure of {G}rid-{T}amari orders.
\newblock {\em J. Combin. Theory Ser. A}, 148:27--56, 2017.

\bibitem[MP17]{MannevillePilaud-accordion}
Thibault Manneville and Vincent Pilaud.
\newblock Geometric realizations of the accordion complex of a dissection.
\newblock Preprint,
  \href{http://arxiv.org/abs/1703.09953}{\texttt{arXiv:1703.09953}}. To appear
  in \emph{Discrete \& Comput.~Geom.}, 2017.

\bibitem[NC01]{Neumann-Coto}
Max Neumann-Coto.
\newblock A characterization of shortest geodesics on surfaces.
\newblock {\em Algebr. Geom. Topol.}, 1(1):349--368, 2001.

\bibitem[OPS18]{OpperPlamondonSchroll}
Sebastian Opper, Pierre-Guy Plamondon, and Sibylle Schroll.
\newblock A geometric model for the derived category of gentle algebras.
\newblock Preprint,
  \href{http://arxiv.org/abs/1801.09659}{\texttt{arXiv:1801.09659}}, 2018.

\bibitem[PPP17]{PaluPilaudPlamondon}
Yann Palu, Vincent Pilaud, and Pierre-Guy Plamondon.
\newblock Non-kissing complexes and tau-tilting for gentle algebras.
\newblock Preprint,
  \href{http://arxiv.org/abs/1707.07574v3}{\texttt{arXiv:1707.07574v3}}, 2017.

\bibitem[PPS18]{PilaudPlamondonStella}
Vincent Pilaud, Pierre-Guy Plamondon, and Salvatore Stella.
\newblock A {$\tau$}-tilting approach to dissections of polygons.
\newblock {\em SIGMA Symmetry Integrability Geom. Methods Appl.}, 14:Paper No.
  045, 8, 2018.

\bibitem[Thu08]{Thurston}
Dylan Thurston.
\newblock Geometric intersection of curves on surfaces.
\newblock Preprint, 2008.

\bibitem[Zie98]{Ziegler-polytopes}
G{\"u}nter~M. Ziegler.
\newblock {\em Lectures on Polytopes}, volume 152 of {\em Graduate Texts in
  Mathematics}.
\newblock Springer-Verlag, New York, 1998.

\end{thebibliography}
\label{sec:biblio}

\end{document}